 \documentclass[11pt]{article}
 \usepackage{amsfonts}
 \usepackage{mathrsfs}
 \usepackage{bbm}
 \usepackage{amsthm}
\usepackage{amsmath,amssymb}
\def\Box{\begin{flushright}\rule{2mm}{2mm}\end{flushright}}

\def\MatrixFont{\bf}
\def\VectorFont{\bf}

\newcommand{\mB}{{\MatrixFont B}}
\newcommand{\mC}{{\MatrixFont C}}
\newcommand{\mD}{{\MatrixFont D}}
\newcommand{\mE}{{\MatrixFont E}}
\newcommand{\mF}{{\MatrixFont F}}

\newcommand{\mH}{{\MatrixFont H}}
\newcommand{\mI}{{\MatrixFont I}}

\newcommand{\mK}{{\MatrixFont K}}

\newcommand{\mM}{{\MatrixFont M}}

\newcommand{\mP}{{\MatrixFont P}}
\newcommand{\mQ}{{\MatrixFont Q}}
\newcommand{\mR}{{\MatrixFont R}}

\newcommand{\mT}{{\MatrixFont T}}

\newcommand{\mX}{{\MatrixFont X}}

\newcommand{\va}{{\VectorFont a}}

\newcommand{\vc}{{\VectorFont c}}

\newcommand{\vv}{{\VectorFont v}}
\newcommand{\vw}{{\VectorFont w}}
\newcommand{\vx}{{\VectorFont x}}
\newcommand{\vy}{{\VectorFont y}}
\newcommand{\vz}{{\VectorFont z}}

\newcommand{\Expectation}{\operatorname{\mathbb{E}}}

\newcommand{\Covariance}{\operatorname{Cov}}

\def\diag{\qopname\relax o{diag}}

\def\trace{\qopname\relax o{tr}}

\newtheorem{theorem}{Theorem}[section]
\newtheorem{lemma}[theorem]{Lemma}

\theoremstyle{definition}

\theoremstyle{remark}
\newtheorem{remark}[theorem]{Remark}

\theoremstyle{algorithm}
\newtheorem{algorithm}[theorem]{Algorithm}

\theoremstyle{corollary}

\theoremstyle{example}
\newtheorem{example}[theorem]{Example}

\usepackage{amssymb}
\usepackage{epsfig}
\usepackage{times}
\title{Sensor Selection  Based on Generalized Information Gain for Target Tracking in Large Sensor Networks}

\author{Xiaojing Shen, Member, IEEE, Pramod K. Varshney, Fellow,
IEEE
\thanks{This work was supported  in part by  U.S. Air Force Office of
Scientific Research (AFOSR) under Grants FA9550-10-1-0263 and
FA9550-10-1-0458 and in part by the NNSF of China \# 61004138  and  IRT1273.} 
\thanks{Xiaojing Shen  (corresponding author, shenxj@scu.edu.cn) and Pramod K. Varshney (varshney@syr.edu) are with the Department of Electrical
Engineering and Computer Science, Syracuse University, NY, 13244,
USA.  Xiaojing Shen (shenxj@scu.edu.cn) is on leave from Department
of Mathematics, Sichuan University, Chengdu, Sichuan 610064,
China.}}

\begin{document}
 \maketitle
\begin{abstract}
In this paper,  sensor selection problems for target tracking in
large sensor networks with linear equality or inequality constraints
are considered. First, we derive an equivalent Kalman filter for
sensor selection, i.e., generalized information filter. Then, under
a regularity condition, we prove that the multistage look-ahead
policy that minimizes either the final or the average estimation
error covariances of next multiple time steps is equivalent to a
myopic sensor selection policy that maximizes the trace of the
generalized information gain at each time step. Moreover, when the
measurement noises are uncorrelated between sensors, the optimal
solution can be obtained analytically for sensor selection when
constraints are temporally separable. When constraints are
temporally inseparable,  sensor selections can be obtained by
approximately solving a linear programming  problem so that the
sensor selection problem for a large sensor network can be dealt
with quickly. Although there is no guarantee that the gap between
the performance of the chosen subset and the performance bound is
always small, numerical examples suggest that the algorithm is
near-optimal in many cases. Finally, when the measurement noises are
correlated between sensors, the sensor selection problem with
temporally inseparable constraints can be relaxed to a Boolean
quadratic programming   problem which can be efficiently solved by a
Gaussian randomization procedure along with solving a semi-definite
programming problem. Numerical examples show that the proposed
method is much better than the method that ignores dependence of
noises.
\end{abstract}

\noindent{\bf keywords:} Sensor selection; generalized information
gain;   sensor networks, target tracking

\section{Introduction}\label{sec_1}

Over the past twenty years, advances in sensor technologies have led
to the emergence of large numbers of low-cost sensing devices with a
fair amount of computing and communication capabilities. Large
sensor networks have attracted much attention both from theoretical
and practical standpoints and have become a fast-growing research
area. To efficiently manage large sensor networks, one typically
designs a policy for determining the optimal sensor network
performance and resource utilization at each time, within logical or
budget constraints. The most comprehensive recent survey on sensor
management is provided in the book
\cite{Hero-Castanon-Cochran-Kastella08}. Discussion on more advances
in this area is available in the recent survey paper
\cite{Hero-Cochran11} and references therein. In this paper, we
concentrate on sensor selection problems in which a subset of
sensors are selected at each time instant while tracking a target
that provides optimal performance--resource usage tradeoffs.

The sensor selection problem arises in various applications,
including  target tracking, e.g.,
\cite{Joshi-Boyd09,Mo-Ambrosino-Sinopoli11}, robotics
\cite{Zhang-Ferrari-Qian09}, and wireless networks
\cite{Zhao-Guibas04}. Sensor selection for the target tracking
problem will be considered here. In the literature, the sensor
selection problem has been formulated for different dynamic systems.
In \cite{Joshi-Boyd09}, the state model was assumed to be
deterministic without noise. A convex optimization procedure was
developed based on a heuristic to solve the problem of selecting $k$
sensors from a set of $m$ sensors. Although no optimality guarantees
could be provided for the solution, numerical experiments showed
that it performed well. Another important contribution comes from
the work reported in \cite{Mo-Ambrosino-Sinopoli11} where the state
model was assumed random with noise and a general objective function
of the sensor selection problem was transformed to a quadratic form
by introducing the gain matrix as an additional decision variable.
However, the resulting optimization problem cannot efficiently take
advantage of the structure of the covariance of measurement noise
such as it being a diagonal matrix in the uncorrelated case. In this
paper, the sensor selection problem formulated by the use of the
Moore-Penrose generalized inverse only relies on Boolean decision
variables without introducing additional decision variables. The
resulting optimization problem can efficiently take advantage of the
structure of the measurement noise and obtain the optimal solution
analytically. Many other excellent results on sensor selection for
state estimation in different situations can be found in, e.g.,
\cite{Gupta-Chung-Hassibi-Murray06,Jiang-Kumar-Garcia03,Hernandez-Kirubarajan-BarShalom04,Tharmarasa-Kirubarajan-Hernandez-Sinha07,
Zhao-Nehorai07,Denzler-Brown02,Kolba-Collins07,Kreucher-Hero-Kastella-Morelande07,
Kreucher-Kastella-Hero05,Masazade-Fardad-Varshney12,Williams-Fisher-Willsky07,Hoffmann-Tomlin10}
and references therein.

Sensor management problems are often considered with different
criteria and objectives. Representative approaches for sensor
management include optimization of estimation error covariance
\cite{Joshi-Boyd09,Mo-Ambrosino-Sinopoli11}, Fisher information
\cite{Hernandez-Kirubarajan-BarShalom04,Tharmarasa-Kirubarajan-Hernandez-Sinha07},
and entropy or mutual information
\cite{Zhang-Ferrari-Qian09,Zhao-Nehorai07,Denzler-Brown02,Kolba-Collins07,Kreucher-Hero-Kastella-Morelande07,Kreucher-Kastella-Hero05,Hoffmann-Tomlin10,Xiong-Svensson02}.
Various functions of the estimation error covariance and Fisher
information matrix, including their determinant and trace, have been
used as reward functions for optimal  sensor management. Several
popular measures, including R$\acute{e}$nyi entropy,
Kullback-Leibler (KL) divergence,  and Hellinger-Battacharya
distance, have been used for the calculation of information gain
between two densities. In this paper,  based on the Moore-Penrose
generalized inverse, we will derive a closed-form expression of
information gain for sensor selection called generalized information
gain whose trace function is taken as the reward function for
optimal sensor selection. When the measurement noises are assumed
independent, the notion  of information measure based on the
information gain has been discussed in the literature, see
\cite{Xiong-Svensson02}.

In this paper,  we consider the problem of state estimation for a
linear dynamic system being monitored by multiple sensors. For
sensor selection, we first derive an equivalent Kalman filter for
sensor selection, i.e., generalized information filter. Then, under
a regularity condition, we prove that the multistage look-ahead
policy that minimizes either the final or the average estimation
error covariance of next $N$ time steps is equivalent to a myopic
sensor selection policy that maximizes the trace of the generalized
information gain at each time step. Thus, trace of the generalized
information gain is defined as a measure of information that the
selected sensors provide at each time step. Moreover, when the
measurement noises are uncorrelated between sensors, the  optimal
solution can be obtained analytically when the constraints are
temporally separable. When the constraints are temporally
inseparable, the solution of the sensor selection problem can be
obtained by approximately solving a linear program (LP) so that
sensor selections for a large sensor network can be performed
quickly. Although there is no guarantee that the gap between the
performance of the chosen subset and the performance bound is always
small, numerical examples suggest that the algorithm is near-optimal
in many cases. Finally, when  the measurement noises are correlated
between sensors, the sensor selection problem when the constraints
are temporally inseparable can be relaxed to a Boolean quadratic
programming (BQP) which can be efficiently solved by a Gaussian
randomization procedure along with solving a semi-definite
programming (SDP) problem which can be solved by interior-point
methods \cite{Boyd-Vandenberghe04}. Numerical examples show that the
proposed method yields solutions that are much better than the
method that ignores dependence.

The rest of the paper is organized as follows. Preliminaries are
given in Section \ref{sec_2}, where the generalized information
filter for sensor selection and multistage sensor selection problems
that minimize either the final or the average estimation error
covariances over the next $N$ time steps are formulated. In Section
\ref{sec_3}, under a regularity condition, we prove that multistage
look-ahead policies are equivalent to the myopic sensor selection
policy that maximizes the trace of the generalized information gain
at each time step. In Section \ref{sec_4}, the case of uncorrelated
measurement noises is considered. The  optimal solution is derived
analytically for sensor selection when the constraints are
temporally separable. When the constraints are temporally
inseparable, the sensor selection scheme is obtained by
approximately solving an LP. In Section \ref{sec_4}, the case of
correlated measurement noises is considered. The sensor selection
problem is relaxed to a BQP which can be efficiently solved by a
Gaussian randomization procedure along with solving an SDP problem.
In Section \ref{sec_5}, numerical examples are given and discussed.
In Section \ref{sec_6}, concluding remarks are provided.

\section{Preliminaries}\label{sec_2}
\subsection{Problem formulation}\label{sec_2_1}
We consider a surveillance region of interest (ROI) that is being
monitored by a sensor field for potential targets crossing the ROI.
The fusion center tracks the target by optimally selecting a fixed
number of sensors from a large sensor network under some logical or
budget constraints. Specifically, we consider a $L$-sensor linear
dynamic system
\begin{eqnarray}%
\label{Eqsm_1}\vx_{k+1}&=&\mF_k\vx_k+\vw_k,\\[3mm]
\label{Eqsm_2}\vy_k^i&=&\mH_k^i\vx_k+\vv_k^i,~~ i=1,2,\ldots,L,\\[3mm]
\label{Eqsm_3}\vz_k^i&=&\gamma_k^i\mH_k^i\vx_k+\gamma_k^i\vv_k^i,~~
i=1,2,\ldots,L,
\end{eqnarray}
where $\vx_{k}\in \mathbb{R}^r$, $\mF_k\in \mathbb{R}^{r\times r}$
is an invertible matrix\footnote{The invertibility of the transition
matrix  can be guaranteed in tracking problems, see
\cite{BarShalom-Chen-Mallick04}.}; $\vy_{k}^i\in \mathbb{R}^{n_i}$,
$\mH_k^i\in \mathbb{R}^{n_i\times r}$, $\{\vw_k\}$ and $\{\vv_k^i\}$
are both temporally uncorrelated with zero means and invertible
covariances $\mQ_k$ and $\mR_k^i$ respectively. The covariance of
the noise $\vv_{k}\triangleq((\vv_{k}^1)',\ldots,(\vv_{k}^L)')'$ is
denoted by $\mR_k\triangleq\Covariance(\vv_k)$ which is assumed
invertible, $\mR_k^{ij}\triangleq\Covariance(\vv_k^i,$ $\vv_k^j)$ so
that $\mR_k^{ii}=\mR_k^{i}$. If the $i$-th sensor is selected, we
let $\gamma_k^i=1$, otherwise $\gamma_k^i=0$ (see, e.g.,
\cite{Mo-Ambrosino-Sinopoli11});
$\gamma_k\triangleq(\gamma_k^1,\ldots,\gamma_k^L)'$.
We shall focus on Equations (\ref{Eqsm_1}) and (\ref{Eqsm_3}) for
sensor selection. The stacked measurement equation is written as
\begin{eqnarray}%
\label{Eqsm_4}\vz_k&=&\tilde{\mH}_k\vx_k+\tilde{\vv}_k,
\end{eqnarray}
where
\begin{eqnarray}%
\label{Eqsm_5}\vz_{k}&\triangleq&((\vz_{k}^1)',\ldots,(\vz_{k}^L)')',\\[3mm]
\label{Eqsm_6}\tilde{\vv}_{k}&\triangleq&((\gamma_k^1\vv_{k}^1)',\ldots,(\gamma_k^L\vv_{k}^L)')',\\[3mm]
\label{Eqsm_7}\tilde{\mH}_k&\triangleq&((\gamma_k^1\mH_k^1)',\ldots,(\gamma_k^L\mH_k^L)')'.
\end{eqnarray}
The covariance  of the noise $\tilde{\vv}_k$ is denoted by
\begin{eqnarray}
\label{Eqsm_07}\tilde{\mR}_k\triangleq\Covariance(\tilde{\vv}_k),~~
\tilde{\mR}_k^{ij}\triangleq\Covariance(\gamma_k^i\vv_k^i,\gamma_k^j\vv_k^j)=\gamma_k^i\gamma_k^j\mR_k^{ij}.
\end{eqnarray}
Moreover, we denote by
$\vz_{1:k}\triangleq(\vz_{1}',\ldots,\vz_{k}')'$,
$\vx_{k|k}\triangleq\Expectation[\vx_k|\vz_{1:k}]$,
$\mP_{k|k}\triangleq\Expectation[(\vx_{k|k}-\vx_k)(\vx_{k|k}-\vx_k)']$.

At time $t_k$, the fusion center has $\gamma_k^i, i=1,\ldots,L$,
$\vx_{k|k}$ and $\mP_{k|k}$ (or measurements $\vz_{1:k}$).
The fusion center  is to design the sensor selection scheme for the
next $N$ time steps.  At time $t_{k+n}$, $m_{k+n}$ sensors will be
selected from $L$ sensors, for $n=1, \ldots,N$. They will send their
measurements, compressed measurements or local estimates to the
fusion center. The fusion center makes the final estimates for the
state at times $t_{k+n}$, $n=1, \ldots,N$. The problem is how to
select $m_{k+n}$ sensors from $L$ sensors (i.e. determine the
Boolean decision variables $\gamma_{k+n}^i, i=1,\ldots,L$), $n=1,
\ldots,N$ that minimize the objective function which is
\begin{itemize}
\item either the final estimation error covariance
\begin{eqnarray}
\label{Eqsm_8} f_1\triangleq\mP_{k+N|k+N},
\end{eqnarray}
\item or the average
estimation error covariance
\begin{eqnarray}
\label{Eqsm_9} f_2\triangleq \frac{1}{N}\sum_{n=1}^N\mP_{k+n|k+n}.
\end{eqnarray}
\end{itemize}
The constraint that $m_{k+n}$ sensors are selected from $L$ sensors,
$n=1, \ldots,N$ induces a constraint that is temporally separable.
Moreover, we shall also consider constraints that are temporally
inseparable, for example, energy constraints.

Note that the objective functions (\ref{Eqsm_8})--(\ref{Eqsm_9}) are
matrices. Matrix optimization considered here is in the sense that
if $x^*$ is an optimal solution, then for an arbitrary feasible
solution $x$, $P(x)\succeq P(x^*)$, i.e., $P(x)-P(x^*)$ is a
positive semi-definite matrix (see, e.g.,
\cite{Boyd-Vandenberghe04}).
\subsection{Equivalent Kalman filter for sensor selection}\label{sec_2_2}
It is well known that the Kalman filter provides the globally
optimal solution if the noises are assumed Gaussian, otherwise it
provides the best linear unbiased estimate. It is recursive no
matter whether the covariances of noises are invertible or not and
is given as follows (see, e.g., \cite{Zhu-Zhou-Shen-Song-Luo12}),
\begin{eqnarray}%
\label{Eqsm_10} \vx_{k+1|k+1}&=&\vx_{k+1|k}+\mK_{k+1}(\vz_{k+1}-\tilde{\mH}_{k+1}\vx_{k+1|k}),\\[3mm]
\label{Eqsm_11}\mP_{k+1|k+1}&=&(\mI-\mK_{k+1}\tilde{\mH}_{k+1})\mP_{k+1|k},%
\end{eqnarray}
where $\mI$ is an identity matrix with compatible dimensions,
\begin{eqnarray}
\label{Eqsm_12}\vx_{k+1|k} &=&\mF_k\vx_{k|k},\\[3mm]
\label{Eqsm_13}\mK_{k+1}&=&\mP_{k+1|k}\tilde{\mH}_{k+1}'(\tilde{\mH}_{k+1}\mP_{k+1|k}\tilde{\mH}_{k+1}'+\tilde{\mR}_{k+1})^{+},\\[3mm]
\label{Eqsm_14}\mP_{k+1|k}&=&\mF_{k}\mP_{k|k}\mF_k'+\mQ_{k}.
\end{eqnarray}
The superscript ``$^+$" means Moore-Penrose generalized inverse
(see, e.g., \cite{BenIsrael-Greville03}) \footnote{Here, the
Moore-Penrose generalized inverse is used since the
$(\tilde{\mH}_{k+1}\mP_{k+1|k}\tilde{\mH}_{k+1}'+\tilde{\mR}_{k+1})$
may not be invertible. The reason is that $\tilde{\mH}_{k+1}$ and
$\tilde{\mR}_{k+1}$ defined by (\ref{Eqsm_7}) and (\ref{Eqsm_07})
for  the sensor selection problem include the decision variables
$\{\gamma_{k+1}^1,\ldots,\gamma_{k+1}^L\}$ which have $L-m_{k+1}$
number of zeros.}. If $\mP_{k+1|k}$, $\mP_{k+1|k+1}$ and
$\tilde{\mR}_{k+1}$ are invertible (for example, the case that all
$L$ sensors are selected), then we have the following equivalent
Kalman filter
\begin{eqnarray}
\label{Eqsm_15}\vx_{k+1|k+1} &=&\mP_{k+1|k+1}
(\mP_{k+1|k}^{-1}\vx_{k+1|k}+
\tilde{\mH}_{k+1}'\tilde{\mR}_{k+1}^{-1}\vz_{k+1}),\\[3mm]
\label{Eqsm_16}\mP_{k+1|k+1}&=&(\mP_{k+1|k}^{-1}+\tilde{\mH}_{k+1}'\tilde{\mR}_{k+1}^{-1}\tilde{\mH}_{k+1})^{-1},
\end{eqnarray}
which is usually called the information filter and
$\tilde{\mH}_{k+1}'\tilde{\mR}_{k+1}^{-1}\tilde{\mH}_{k+1}$ is
called the information gain (see e.g., \cite{BarShalom-Li95}). Once
the sensors are selected, the covariance of the noise vector of the
selected sensors is invertible. Note that we assumed that $\mQ_{k},
k=1,2,\ldots,$ are invertible, and it is easy to check that
$\mP_{k+1|k}$ and $\mP_{k+1|k+1}$ which are updated by the Kalman
Filter based on the selected sensors are also invertible. Here,
however, $\tilde{\mR}_{k+1}$ are not invertible, since there are
$L-m_{k+1}$ number of zeros in the decision variables
$\{\gamma_{k+1}^1,\ldots,\gamma_{k+1}^L\}$. Thus, we first prove
that, for the dynamic system (\ref{Eqsm_1}) and (\ref{Eqsm_4})
defined under sensor selection where $\tilde{\mR}_{k+1}$ are not
invertible, there still exists an equivalent Kalman filter similar
to (\ref{Eqsm_15})--(\ref{Eqsm_16}).

\begin{theorem}\label{thm_1}
For the dynamic system defined by (\ref{Eqsm_1}) and (\ref{Eqsm_4})
under sensor selection,
we have the following equivalent Kalman filter (generalized
information filter)
\begin{eqnarray}
\label{Eqsm_17}\vx_{k+1|k+1} &=&\mP_{k+1|k+1}
(\mP_{k+1|k}^{-1}\vx_{k+1|k}+ \tilde{\mH}_{k+1}'\tilde{\mR}_{k+1}^{+}\vz_{k+1}),\\[3mm]
\label{Eqsm_18}\mP_{k+1|k+1}&=&(\mP_{k+1|k}^{-1}+\tilde{\mH}_{k+1}'\tilde{\mR}_{k+1}^{+}\tilde{\mH}_{k+1})^{-1}.
\end{eqnarray}
\end{theorem}
\begin{proof}
See appendix.
\end{proof}

The key difference in Theorem \ref{thm_1} is that
$\tilde{\mR}_{k+1}^{-1}$ where $\tilde{\mR}_{k+1}$ is invertible in
(\ref{Eqsm_15})--(\ref{Eqsm_16}) has been replaced by
$\tilde{\mR}_{k+1}^{+}$ in (\ref{Eqsm_17})--(\ref{Eqsm_18}). Due to
this difference,
$\tilde{\mH}_{k+1}'\tilde{\mR}_{k+1}^{+}\tilde{\mH}_{k+1}$ will be
called \emph{generalized information gain}.  Notice that
$\mP_{k+1|k+1}$ is a function of
$\gamma_{k+1}^1,\ldots,\gamma_{k+1}^L$, since $\tilde{\mH}_{k+1}$
and $\tilde{\mR}_{k+1}$ are functions of
$\gamma_{k+1}^1,\ldots,\gamma_{k+1}^L$. Thus, it is denoted by
$\mP_{k+1|k+1}(\gamma_{k+1}^1,\ldots,\gamma_{k+1}^L)$. Similarly,
$\mP_{k+n|k+n}(\gamma_{k+1}^1,\ldots,\gamma_{k+1}^L,\ldots$,
$\gamma_{k+n}^1,\ldots,\gamma_{k+n}^L)$ is a function of
$\gamma_{k+1}^1$, $\ldots,\gamma_{k+1}^L$, $\ldots$,
$\gamma_{k+n}^1$, $\ldots$, $\gamma_{k+n}^L$, for $n=1,\ldots,N$.

It is the generalized information filter based on Moore-Penrose
generalized inverse that helps us decouple the multistage look-ahead
policies to an equivalent myopic sensor selection policy that
maximizes the generalized information gain with a lower
computational complexity in Section \ref{sec_3}. Another advantage
is that the  sensor selection problem formulated by the use of the
Moore-Penrose generalized inverse only relies on Boolean decision
variables without introducing additional decision variables and can
efficiently take advantage of the structure of the measurement noise
to obtain the optimal solution and efficient algorithms.

\subsection{Optimization problems for sensor selection} \label{sec_2_3}

Thus, by using Theorem \ref{thm_1}, the two sensor selection
problems can be stated as
\begin{eqnarray}%
\nonumber \min_{\gamma_{k+n}^i} ~&&\mP_{k+N|k+N}(\gamma_{k+1}^1,\ldots,\gamma_{k+1}^L,\ldots,\gamma_{k+N}^1,\ldots,\gamma_{k+N}^L)\\
\label{Eqsm_19}&&=(\mP_{k+N|k+N-1}^{-1}+\tilde{\mH}_{k+N}'\tilde{\mR}_{k+N}^{+}\tilde{\mH}_{k+N})^{-1}\\
\mbox{subject to}~
\label{Eqsm_20}&&\sum_{i=1}^L\gamma_{k+n}^i=m_{k+n},~ n=1,\ldots,N,\\
\label{Eqsm_21}&&\gamma_{k+n}^i\in \{0,1\}, i=1,2,\ldots,L,~
n=1,\ldots,N,
\end{eqnarray}
and
\begin{eqnarray}%
\nonumber  \min_{\gamma_{k+n}^i} ~&&\sum_{n=1}^N\mP_{k+n|k+n}(\gamma_{k+1}^1,\ldots,\gamma_{k+1}^L,\ldots,\gamma_{k+N}^1,\ldots,\gamma_{k+N}^L)\\
\label{Eqsm_22}&&=\sum_{n=1}^N(\mP_{k+n|k+n-1}^{-1}+\tilde{\mH}_{k+n}'\tilde{\mR}_{k+n}^{+}\tilde{\mH}_{k+n})^{-1}\\
\mbox{subject to}~
\nonumber &&\sum_{i=1}^L\gamma_{k+n}^i=m_{k+n},~ n=1,\ldots,N,\\
\nonumber &&\gamma_{k+n}^i\in \{0,1\}, i=1,2,\ldots,L,~ n=1,\ldots,
N.
\end{eqnarray}

\section{Generalized Information Measure for Sensor Selection}\label{sec_3}
In this section, we consider some properties of the optimization
problems presented in Section \ref{sec_2_3} that will simplify their
solution. We will show that if the primal sensor selection problem
(\ref{Eqsm_19}) has an optimal solution, then both the problem
(\ref{Eqsm_19}) and the problem (\ref{Eqsm_22}) can be transformed
to equivalent optimization problems that maximize an information
measure at each time step.

\begin{lemma}\label{lem_1}
Consider two optimization problems:
\begin{eqnarray}%
\label{Eqsm_2_1}(A_1)~~~~~\max_{\vx\in\mathcal {S}} ~&&\mM(\vx),\\
\label{Eqsm_2_2}(A_2)~~~~~\max_{\vx\in\mathcal {S}}
~&&\trace(\mM(\vx)),
\end{eqnarray}
where $\mM(\vx)$ is a matrix for an arbitrary $\vx\in\mathcal {S}$;
$\mathcal {S}$ specifies the constraint on the decision variable
$\vx$. If the problem ($A_1$) has an optimal solution, then the
problem ($A_1$) is equivalent to ($A_2$).
\end{lemma}
\begin{proof}
See appendix.
\end{proof}

\begin{lemma}\label{lem_2}
Consider two optimization problems:
\begin{eqnarray}%
\label{Eqsm_2_5}&&(B_1)~~~~~\min_{\vx_i\in\mathcal {S}_i,
i=1,\ldots,n}
~\mM_n(\vx_1,\ldots,\vx_n),~ for ~n=1,\ldots, N,\\[3mm]
\label{Eqsm_2_6}&&(B_2)~~~~~\min_{\vx_n\in\mathcal {S}_n,
n=1,\ldots, N,}~ \sum_{n=1}^N\mM_n(\vx_1,\ldots,\vx_n)
\end{eqnarray}
where $\mM_n(\vx_1,\ldots,\vx_n)$ is a function of decision
variables $\vx_1,\ldots,\vx_n$, for $n=1,\ldots, N$. If the optimal
solution that minimizes $\mM_n(\vx_1,\ldots,\vx_n)$,
$(\vx_1^*,\ldots,\vx_n^*)$, is the same as the one that minimizes
$\mM_{n+1}(\vx_1,\ldots,\vx_{n+1})$, for $n=1,\ldots, N-1$, then the
optimal solution that minimizes $\mM_N(\vx_1,\ldots,\vx_N)$ (($B_1$)
with $n=N$) is the same as that for ($B_2$).
\end{lemma}
\begin{proof}
See appendix.
\end{proof}

Based on Lemma  \ref{lem_2}, the solution to both the problem
(\ref{Eqsm_19}) and the problem (\ref{Eqsm_22}) can be simplified
and obtained by solving $N$ optimization problems separately.
\begin{lemma}\label{lem_3}
If  the primal sensor selection problem (\ref{Eqsm_19}) has an
optimal solution, then both the problem (\ref{Eqsm_19}) and the
problem (\ref{Eqsm_22}) can be transformed to the equivalent problem
that solves $N$ optimization problems that maximize
$\tilde{\mH}_{k+n}'\tilde{\mR}_{k+n}^{+}\tilde{\mH}_{k+n}$,
$n=1,\ldots,N$ respectively, i.e.,
\begin{eqnarray}%
\label{Eqsm_2_7} \max_{\gamma_{k+n}^i} ~&&\tilde{\mH}_{k+n}'\tilde{\mR}_{k+n}^{+}\tilde{\mH}_{k+n}~~~~~ \mbox{for}~ n=1,\ldots, N,\\
\mbox{subject to}~
\nonumber&&\sum_{i=1}^L\gamma_{k+n}^i=m_{k+n},\\
\nonumber&&\gamma_{k+n}^i\in \{0,1\}, i=1,2,\ldots,L,
\end{eqnarray}
where $\tilde{\mH}_{k+n}$ and $\tilde{\mR}_{k+n}$ are defined in
Equations (\ref{Eqsm_7}) and (\ref{Eqsm_07}) respectively; the
superscript ``$^+$" indicates Moore-Penrose generalized inverse
\cite{BenIsrael-Greville03}. That is, the problems (\ref{Eqsm_19}),
(\ref{Eqsm_22}) and (\ref{Eqsm_2_7}) have the same optimal solution.
\end{lemma}
\begin{proof}
See appendix.
\end{proof}
\begin{remark}\label{rmk_1}
Lemma \ref{lem_3} shows that multistage look-ahead policies, i.e.,
the problem (\ref{Eqsm_19}) and the problem (\ref{Eqsm_22}), are
equivalent to a myopic sensor selection policy that maximizes the
generalized information gain with a lower computational complexity.
Why do the problem (\ref{Eqsm_19}) and the problem (\ref{Eqsm_22})
with different objectives have the same optimal solution? The main
reason is that the objectives and constraints are temporally
separable. For example, consider $m_{k+n}=1$, i.e., select one
sensor at each time step, if there is a sensor which has the
smallest noise and provides the most information at each time step,
then the selection of the sensor at each time step is the optimal
sensor selection scheme no matter whether the objective is the final
estimation error covariance or the average estimation error
covariance.
\end{remark}

Moreover, based on Lemmas \ref{lem_1} and \ref{lem_3}, we have the
following theorem.

\begin{theorem}\label{thm_2}
If  the primal sensor selection problem (\ref{Eqsm_19}) has an
optimal solution, both the problem (\ref{Eqsm_19}) and the problem
(\ref{Eqsm_22}) can be transformed to the equivalent problem
requiring the solution of $N$ optimization problems that maximize
$\trace(\tilde{\mH}_{k+n}'\tilde{\mR}_{k+n}^{+}\tilde{\mH}_{k+n})$,
$n=1,\ldots,N$ respectively, i.e.,
\begin{eqnarray}%
\label{Eqsm_2_07} \max_{\gamma_{k+n}^i} ~&&\trace(\tilde{\mH}_{k+n}'\tilde{\mR}_{k+n}^{+}\tilde{\mH}_{k+n})~~~~~ \mbox{for}~ n=1,\ldots, N,\\
\mbox{subject to}~
\nonumber&&\sum_{i=1}^L\gamma_{k+n}^i=m_{k+n},\\
\nonumber&&\gamma_{k+n}^i\in \{0,1\}, i=1,2,\ldots,L,
\end{eqnarray}
where  $\tilde{\mH}_{k+n}$ and $\tilde{\mR}_{k+n}$ are defined in
Equations (\ref{Eqsm_7}) and (\ref{Eqsm_07}) respectively. That is,
the problems (\ref{Eqsm_19}), (\ref{Eqsm_22}) and (\ref{Eqsm_2_07})
have the same optimal solution.
\end{theorem}
\begin{remark}\label{rmk_2}
Theorem \ref{thm_2} shows that both the minimization of the final
estimation error covariance and  minimization of the average
estimation error covariance are equivalent to maximization of the
trace of the generalized information gain of each time step. Thus,
the objective function
\begin{eqnarray}
\label{Eqsm_019}
\trace(\tilde{\mH}_{k+n}'\tilde{\mR}_{k+n}^{+}\tilde{\mH}_{k+n})
\end{eqnarray}
of the problem (\ref{Eqsm_2_7}), i.e., trace of the generalized
information gain, is defined as a measure of information that the
selected sensors  provide at $(k+n)$-th time step. Determinant of
the generalized information gain can be similarly defined as a
measure of information if
$\tilde{\mH}_{k+n}'\tilde{\mR}_{k+n}^{+}\tilde{\mH}_{k+n}$ is a
positive definite matrix.
When the measurement noises are assumed independent, more
information measures based on  information gain can be formulated,
see e.g., \cite{Xiong-Svensson02}.

Furthermore, the information measure (\ref{Eqsm_019}) has an
advantage that it does not depend on pdfs of the states and
measurements, but only relies on covariances of noises and the
measurement matrices. Maximizing this measure can be employed as an
alternative criterion for sensor selection, which will be used in
sensor selection problems when the constraints are temporally
inseparable  in the following sections. When pdfs are known, it is
better to try to use the information criteria based on pdfs such as
optimization of Fisher information, entropy or mutual information
for sensor selection (see, e.g., \cite{Hero-Cochran11}).

\end{remark}

\section{Sensor Selection Schemes for Uncorrelated Sensor Measurement Noises}\label{sec_4}
\subsection{Optimal Sensor Selection Scheme for Temporally Separable
Constraints}\label{sec_4_1}
%
When sensor measurement noises are uncorrelated and the constraints
are temporally separable, we have the following result that defines
the optimal sensor selection scheme.

\begin{theorem}\label{thm_3}
Let the information measure corresponding to the $i$-th sensor at
$(k+n)$-th time be denoted as
$a_{k+n}^i\triangleq\trace({(\mH_{k+n}^i)}'$
$(\mR_{k+n}^i)^{-1}\mH_{k+n}^i), i=1,\ldots,L$.
Let $\{a_{k+n}^{i_1},\ldots$, $a_{k+n}^{i_L}\}$ denote
$\{a_{k+n}^1,\ldots$, $a_{k+n}^L\}$ rearranged in descending order.
If the problem (\ref{Eqsm_19}) has an optimal solution, then the
optimal sensor selection scheme for both  the problem
(\ref{Eqsm_19}) and the problem (\ref{Eqsm_22}) is
$\gamma_{k+n}^{i_1}=1,\ldots,\gamma_{k+n}^{i_{m_{k+n}}}=1$,
$\gamma_{k+n}^{i_{m_{k+n+1}}}=0,\ldots,\gamma_{k+n}^{i_L}=0$, for
$n=1,\ldots,N$. The optimality of sensor selection scheme is in the
sense of either the minimization of the covariance of the final
estimation error (\ref{Eqsm_8}) or the average estimation error
(\ref{Eqsm_9}) or maximization of the information measure
(\ref{Eqsm_019}). If the problem (\ref{Eqsm_19}) does not have an
optimal solution, the optimality of the sensor selection scheme is
only in the sense of maximization of the information measure
(\ref{Eqsm_019}).
%
\end{theorem}
\begin{proof}
If the measurement noises are uncorrelated between sensors, then
$\tilde{\mR}_{k+n}$ is a block diagonal matrix with
$\tilde{\mR}_{k+n}^{ij}=0, i\neq j$. Thus,
\begin{eqnarray}
\label{Eqsm_new_3_1}\tilde{\mR}_{k+n}^{+}=(\diag(\gamma_{k+n}^1R_{k+n}^1,\ldots,\gamma_{k+n}^LR_{k+n}^L))^+
=\diag(\gamma_{k+n}^1(R_{k+n}^1)^{-1},\ldots,\gamma_{k+n}^L(R_{k+n}^L)^{-1})
\end{eqnarray}
which follows from the definition of Moore-Penrose generalized
inverse. Moreover, by Theorem \ref{thm_2}, the problem
(\ref{Eqsm_19}) is equivalent to
\begin{eqnarray}%
\nonumber\max_{\gamma_{k+n}^i} ~&&\trace(\tilde{\mH}_{k+n}'\tilde{\mR}_{k+n}^{+}\tilde{\mH}_{k+n})~~~~~ \mbox{for}~ n=1,\ldots, N,\\
\nonumber&&=\sum_{i=1}^L \gamma_{k+n}^i\trace({(\mH_{k+n}^i)}'({\mR_{k+n}^i})^{-1}\mH_{k+n}^i)\\
\mbox{subject to}~
\nonumber&&\sum_{i=1}^L\gamma_{k+n}^i=m_{k+n},\\
\nonumber&&\gamma_{k+n}^i\in \{0,1\}, i=1,2,\ldots,L.
\end{eqnarray}
If we define
$a_{k+n}^i\triangleq\trace({(\mH_{k+n}^i)}'(\mR_{k+n}^i)^{-1}\mH_{k+n}^i),
i=1,\ldots,L$ and $\{a_{k+n}^{i_1},\ldots$, $a_{k+n}^{i_L}\}$
denotes $\{a_{k+n}^1,\ldots$, $a_{k+n}^L\}$  rearranged in
descending order, then the optimal solution is
$\gamma_{k+n}^{i_1}=1,\ldots,\gamma_{k+n}^{i_{m_{k+n}}}=1$,
$\gamma_{k+n}^{i_{m_{k+n+1}}}=0,\ldots,\gamma_{k+n}^{i_L}=0$.
\Box
\end{proof}

\subsection{Extension to Temporally Inseparable
Constraints}\label{sec_4_2}

Many constraints on sensor selection can be represented as linear
equalities or inequalities such as logical constraints and budget
constraints (see, e.g.,
\cite{Joshi-Boyd09,Mo-Ambrosino-Sinopoli11}).
Let us denote linear equalities or inequalities as follows
\begin{eqnarray}%
\label{Eqsm_new_3_2}\va_p'\gamma ~\unrhd~b_p,~ p=1,\ldots,P,
\end{eqnarray}
where
\begin{eqnarray}
\label{Eqsm_new_3_3}\gamma\triangleq(\gamma_{k+1}',\ldots,\gamma_{k+N}')',
~~\gamma_{k+n}\triangleq(\gamma_{k+n}^1,\ldots,\gamma_{k+n}^L)',
~\mbox{for}~  n=1,\ldots, N;
\end{eqnarray}
$\va_p$ is a vector with a compatible dimension and $b_p$ is a
scalar; ``$\unrhd$" can represent either ``$\geq$" ``$\leq$" or
``$=$" for each $n$. In general, these constraints are temporally
inseparable,  which makes the optimization problems with objectives
(\ref{Eqsm_19}) and (\ref{Eqsm_22}) not separable and highly
nonlinear in variables $\gamma_{k+n}^i$. The corresponding
optimization problems are very hard to solve.

However, from Remark \ref{rmk_2}, the trace of generalized
information ,
$\trace(\tilde{\mH}_{k+n}'\tilde{\mR}_{k+n}^{+}\tilde{\mH}_{k+n})$,
can be defined as the measure of information that  selected sensors
provide. Thus, we can try to  maximize the available information
gain from time $k+1$ to $k+N$ by optimizing the selection of sensors
so that better estimation performance can be expected. We shall
consider the following objective (i.e, sum of the weighted
information measure)
\begin{eqnarray}%
\label{Eqsm_new_3_4}f_3&\triangleq&\sum_{n=1}^N\omega_n\trace(\tilde{\mH}_{k+n}'\tilde{\mR}_{k+n}^{+}\tilde{\mH}_{k+n}),
\end{eqnarray}
where $\omega_n, n=1,\ldots,N$ are weights which place different
importance on different time steps. For example, if the state
estimation at the final time is more important, a larger weight
$\omega_N$ can be used. If the state estimation of each time step is
equally important, an equal weight structure
$\omega_1=\cdots=\omega_N$
 can be used. Therefore, we consider the following
optimization problem for sensor selection:
\begin{eqnarray}%
\label{Eqsm_new_3_5}\max_{\gamma_{k+n}^i} ~&&\sum_{n=1}^N\omega_n\trace(\tilde{\mH}_{k+n}'\tilde{\mR}_{k+n}^{+}\tilde{\mH}_{k+n}) \\
\nonumber\mbox{subject to}~
~ &&\va_p'\gamma ~\unrhd~b_p,~ p=1,\ldots,P,\\
\nonumber&&\gamma_{k+n}^i\in \{0,1\}, i=1,2,\ldots,L, n=1,\ldots, N.
\end{eqnarray}


Since sensor measurement noises are assumed uncorrelated in this
section, by Equation (\ref{Eqsm_new_3_1}), the problem
(\ref{Eqsm_new_3_5}) is equivalent to
\begin{eqnarray}%
\label{Eqsm_new_3_6}\max_{\gamma_{k+n}^i} ~&&\sum_{n=1}^N\omega_n\sum_{i=1}^L \gamma_{k+n}^i\trace({(\mH_{k+n}^i)}'({\mR_{k+n}^i})^{-1}\mH_{k+n}^i)\\
\nonumber\mbox{subject to}~
~ &&\va_p'\gamma ~\unrhd~b_p,~ p=1,\ldots,P,\\
\nonumber&&\gamma_{k+n}^i\in \{0,1\}, i=1,2,\ldots,L, n=1,\ldots,N,
\end{eqnarray}
which is a Boolean linear programming (BLP) problem and  the optimal
objective function value is denoted by $f_{BLP}^*$. It can be
relaxed by replacing the nonconvex constraints $\gamma_{k+n}^i\in
\{0,1\}$ with the convex constraints $0\leq\gamma_{k+n}^i\leq 1$,
$i=1,2,\ldots,L, n=1,\ldots,N$. Thus, we have the following LP
problem:
\begin{eqnarray}%
\label{Eqsm_new_3_7}\max_{\gamma_{k+n}^i} ~&&\trace(\Gamma\mD')\\
\nonumber\mbox{subject to}~
~ &&\va_p'\gamma ~\unrhd~b_p,~ p=1,\ldots,P,\\
\nonumber&&0\leq \gamma_{k+n}^i\leq 1, i=1,2,\ldots,L, n=1,\ldots,N,
\end{eqnarray}
where $\gamma$ is defined by (\ref{Eqsm_new_3_3}); $\Gamma$ is a
$L\times N$ matrix with $i$-th row and $n$-th column element being
$\gamma_{k+n}^i$, i.e,
\begin{eqnarray}%
\label{Eqsm_new_3_8}\Gamma&\triangleq&  \left(
                     \begin{array}{cccc}
                       \gamma_{k+1}^1 & \gamma_{k+2}^1 & \cdots& \gamma_{k+N}^1 \\
                      \gamma_{k+1}^2 &\gamma_{k+2}^2& \cdots& \gamma_{k+N}^2 \\
                       \vdots& \vdots& \ddots& \vdots\\
                      \gamma_{k+1}^L& \gamma_{k+2}^L & \cdots& \gamma_{k+N}^L\\
                     \end{array}
                   \right)
\end{eqnarray}
and $\mD$ is a $L\times N$ matrix
\begin{eqnarray}%
\label{Eqsm_new_3_9}\mD&\triangleq& \left(d_{in}
        \right)_{L\times N},~~ d_{in}\triangleq \omega_n\trace({(\mH_{k+n}^i)}'({\mR_{k+n}^i})^{-1}\mH_{k+n}^i).
\end{eqnarray}

It is well known that LP problems can be solved efficiently. The
solution of the problem (\ref{Eqsm_new_3_7}) is denoted by
$(\gamma_{k+n}^i)_{LP}^*$, $i=1,\ldots,L, n=1,\ldots,N$. The
corresponding objective function is denoted by $f_{LP}^*$. Note that
the feasible solution set of the problem (\ref{Eqsm_new_3_7})
contains that of the problem (\ref{Eqsm_new_3_6}) so that
$f_{BLP}^*\leq f_{LP}^*$. Based on $(\gamma_{k+n}^i)_{LP}^*$, we can
generate a suboptimal feasible solution of the problem
(\ref{Eqsm_new_3_6}) denoted by $\hat{\gamma}_{k+n}^i$,
$i=1,\ldots,L, n=1,\ldots,N$. The corresponding objective function
is denoted by $\hat{f}_{BLP}$ and $\hat{f}_{BLP}\leq f_{BLP}^*$. The
difference $g=f_{LP}^*-\hat{f}_{BLP}$ is called the gap in
\cite{Joshi-Boyd09}. The gap is useful in evaluating the performance
of the suboptimal solution $\hat{\gamma}_{k+n}^i$. We can say
$\hat{\gamma}_{k+n}^i$ is no more than $g$-suboptimal.

Note that the procedure of generating a feasible solution of the
problem (\ref{Eqsm_new_3_6})  from $(\gamma_{k+n}^i)_{LP}^*$ is
problem dependent, i.e., relying on the equalities or inequalities
(\ref{Eqsm_new_3_2}) and the Boolean constraint. As an illustration,
let us consider a representative example. Besides the temporally
separable constraints (\ref{Eqsm_20}) and (\ref{Eqsm_21}), we
consider an energy constraint  which is temporally inseparable as
follows
\begin{eqnarray}
\label{Eqsm_new_3_10}\sum_{n=1}^N\gamma_{k+n}^i\leq m_k^i,~
i=1,\ldots,L,
\end{eqnarray}
which means that the $i$-th sensor can only be selected $m_k^i$
times from time $k+1$ to time $k+N$ ($m_k^i<N$), for $
i=1,\ldots,L$. Thus, the specific form of the optimization problem
(\ref{Eqsm_new_3_7}) with the constraints (\ref{Eqsm_20}) and
(\ref{Eqsm_new_3_10}) can be represented to
\begin{eqnarray}%
\label{Eqsm_new_3_11}\max_{\gamma_{k+n}^i} ~&&\trace(\Gamma \mD')\\
\nonumber\mbox{subject to}~
~ &&\va_p'\gamma ~\unrhd~b_p,~ p=1,\ldots,P,\\
\nonumber&&0\leq \gamma_{k+n}^i\leq 1, i=1,2,\ldots,L, n=1,\ldots,N,
\end{eqnarray}
where $\Gamma$ and $\gamma$ are defined by (\ref{Eqsm_new_3_8}) and
(\ref{Eqsm_new_3_3}) respectively; $P=N+L$,
\begin{eqnarray}%
\nonumber\va_p&\triangleq&(\vc_{1,p}',\ldots,\vc_{N,p}')',
~\vc_{n,p}\triangleq\left\{
            \begin{array}{ll}
              \textbf{1}, & \hbox{$p=n$,} \\
              \textbf{0}, & \hbox{$p\neq n$, }
            \end{array}
          \right. n=1\ldots,N,\\
\label{Eqsm_new_3_12}b_p&\triangleq& m_{k+p}, ~``\unrhd" \mbox{means} ``=",\\
\nonumber&&~~\mbox{for}~ p=1,\ldots, N,~ (\mbox{corresponds to the constraints (\ref{Eqsm_20})}); \\[3mm]
\nonumber\va_{p}&\triangleq&(\textbf{1}_i',\ldots,\textbf{1}_i')',   \\
\label{Eqsm_new_3_012}b_p&\triangleq& m_{k}^i,~ ``\unrhd"
\mbox{means} ``\leq",\\
\nonumber&& ~~\mbox{for}~ p= N+i,~ i=1,\ldots,L, (\mbox{corresponds
to the constraints (\ref{Eqsm_new_3_10})}),
\end{eqnarray}
where \textbf{1} and \textbf{0} denote  $L$-dimensional vectors with
1 entries and  0 entries  respectively  and $\textbf{1}_i$ means an
$L$-dimensional vector whose  $i$-th entry is 1 others are 0s.

The sensor selection scheme with the energy constraint for
uncorrelated sensors is described by the following algorithm.

\begin{algorithm}[Sensor selection scheme with the energy constraint for uncorrelated sensors]\label{alg_2}
~
\begin{itemize}

\item Step 1: Given an optimal solution of (\ref{Eqsm_new_3_11}) $(\gamma_{k+n}^i)_{LP}^*$, obtain the optimal objective function $f_{LP}^*$.

\item Step 2: Generate a feasible solution of the problem (\ref{Eqsm_new_3_6}) with
the constraints (\ref{Eqsm_20}), (\ref{Eqsm_21}) and
(\ref{Eqsm_new_3_10}) from $(\gamma_{k+n}^i)_{LP}^*$ as follows.

We generate the feasible solution based on the importance (weight)
of the information of each time step. Without loss of generality,
assume that $\omega_1\leq\omega_2\leq\ldots\leq\omega_N$. Thus, we
generate the selection scheme from the $N$-th time step to the first
time step. Set the index set of candidate sensors
$\mathbbm{i}\triangleq\{1,\ldots,L\}$.
\begin{itemize}
%
\item Iteratively generate $\hat{\gamma}_{k+N-n}^i$ for the $(N-n)$-th time step, $n=0,\ldots,(N-1)$ as follows:

for~~ $n=0:(N-1)$
\begin{eqnarray}%
\label{Eqsm_new_3_15}\hat{\gamma}_{k+N-n}^i &\triangleq&\left\{
                            \begin{array}{ll}
                              1, & \hbox{if~ $i\in \mathbbm{i}_1$,} \\
                              0, & \hbox{if~ $i\in \mathbbm{i}_2$,}
                            \end{array}
                          \right., for~ i=1,\ldots,L,
\end{eqnarray}
where $\mathbbm{i}_1$  is the index set of the first $m_{k+N-n}$
maximum entries of ($(\gamma_{k+N-n}^1)_{LP}^* ,\ldots$,
$(\gamma_{k+N-n}^L)_{LP}^*$) in the index set of candidate sensors
$\mathbbm{i}$ and denote $\mathbbm{i}_2=\mathbbm{i}-\mathbbm{i}_1$.
Set $m_k^i:=m_k^i-1$, for $i\in\mathbbm{i}_1$.
 Update the index set of candidate sensors
$\mathbbm{i}\triangleq\{i:m_k^i>0, i=1,\ldots,L\}$.

end
\end{itemize}

\item Step 3: Output  $g$-suboptimal solution $\hat{\gamma}_{k+n}^i, i=1,\ldots,L, n=1,\ldots,N$  and the corresponding
objective $\hat{f}_{BLP}$, where $g=f_{LP}^*-\hat{f}_{BLP}$ is the
gap.
\end{itemize}
\end{algorithm}
Here, to construct the feasible solution satisfying the constraints
(\ref{Eqsm_20}), (\ref{Eqsm_21}) and (\ref{Eqsm_new_3_10}), we
employ the equation
(\ref{Eqsm_new_3_15}). The main computation complexity is in Step 1
where an LP problem needs to be solved. Illustrative examples will
be presented in Section \ref{sec_6}.

%
%

\section{Sensor Selection Schemes for Correlated Sensor Measurement Noises}\label{sec_5}


In this section, for correlated sensor measurement noises, we again
determine the sensor selection scheme by maximizing the weighted
information measure:
\begin{eqnarray}%
\label{Eqsm_new_4_1}\max_{\gamma_{k+n}^i} ~&&\sum_{n=1}^N\omega_n\trace(\tilde{\mH}_{k+n}'\tilde{\mR}_{k+n}^{+}\tilde{\mH}_{k+n}) \\
\nonumber\mbox{subject to}~
~ &&\va_p'\gamma ~\unrhd~b_p,~ p=1,\ldots,P,\\
\nonumber&&\gamma_{k+n}^i\in \{0,1\}, i=1,2,\ldots,L.
\end{eqnarray}
where the linear constraints are defined in (\ref{Eqsm_new_3_2})
that may include both the temporally separable and inseparable
constraints. Since sensor measurement noises are correlated,
to obtain the optimal solution, an exhaustive search is necessary
since $\tilde{\mR}_{k+n}$ has no special structure such as it being
a diagonal matrix. For the simplest case of the temporally separable
constraint (\ref{Eqsm_20}) and $N=1$, there are a total of
$\left(\begin{array}{c}
                                                                                                                  L \\
                                                                                                                  m_{k+n}
                                                                                                                \end{array}\right)$
feasible solutions. For large $L$ and $m_{k+n}$, such an exhaustive
search may not be feasible in real time.
Thus, to make the solution computationally more efficient, the
problem (\ref{Eqsm_new_4_1}) is approximately solved by replacing
$\tilde{\mR}_{k+n}^{+}$ by $\mR_{k+n}^{-1}$. This approximation is
lossless for the case of uncorrelated sensor noises and temporally
separable constraint (i.e., does not change the optimal solution in
Theorem \ref{thm_3}). More discussion on approximation loss for
different dependences will be given in  Section \ref{sec_6}. Thus,
we consider the approximate problem
\begin{eqnarray}%
\label{Eqsm_new_4_2}\max_{\gamma_{k+n}^i} ~&&\sum_{n=1}^N\omega_n\trace(\tilde{\mH}_{k+n}'\mR_{k+n}^{-1}\tilde{\mH}_{k+n}) \\
\nonumber\mbox{subject to}~
~ &&\va_p'\gamma ~\unrhd~b_p,~ p=1,\ldots,P,\\
\nonumber&&\gamma_{k+n}^i\in \{0,1\}, i=1,2,\ldots,L.
\end{eqnarray}
Moreover, from the definition of $\tilde{\mH}_{k+n}$ (\ref{Eqsm_7}),
we have
\begin{eqnarray}%
\nonumber&&\trace(\tilde{\mH}_{k+n}'\mR_{k+n}^{-1}\tilde{\mH}_{k+n})\\
\nonumber&=&\trace(\sum_{i=1}^L\sum_{s=1}^L\gamma_{k+n}^i\gamma_{k+n}^s(\mH_{k+n}^i)'\mT_{k+n}^{is}\mH_{k+n}^s)\\
\nonumber&=&\sum_{i=1}^L\sum_{s=1}^L\gamma_{k+n}^i\gamma_{k+n}^s\trace((\mH_{k+n}^i)'\mT_{k+n}^{is}\mH_{k+n}^s)\\
\label{Eqsm_new_4_3}&=&-\gamma_{k+n}'\mB_{k+n}\gamma_{k+n},
\end{eqnarray}
where $\mT_{k+n}^{is}$ is the $i$-th row block and $s$-th column
block of the matrix $\mT_{k+n}$, $\mT_{k+n}\triangleq
\mR_{k+n}^{-1}$; the $i$-th row and $s$-th column of $\mB_{k+n}$ is
$-\trace((\mH_{k+n}^i)'\mT_{k+n}^{is}\mH_{k+n}^s)$. Thus, the
problem is equivalent to solving the following Boolean quadratic
programming (BQP) problem
\begin{eqnarray}%
\label{Eqsm_new_4_4}\min_{\gamma_{k+n}^i} ~&&\sum_{n=1}^N\omega_n\gamma_{k+n}'\mB_{k+n}\gamma_{k+n} \\
\nonumber\mbox{subject to}~
~ &&\va_p'\gamma ~\unrhd~b_p,~ p=1,\ldots,P,\\
\nonumber&&\gamma_{k+n}^i\in \{0,1\}, i=1,2,\ldots,L, n=1,\ldots, N.
\end{eqnarray}
For this problem, however, it is still hard to obtain an optimal
solution, since the nonconvex Boolean constraints and $\mB_{k+n}$
may not be a positive semi-definite matrix. It is known to belong to
the class of NP-hard problems. Fortunately, this class of problems
can be solved by a recently developed  computationally efficient
approximation technique (see, e.g., \cite{Luo-Ma-So-Ye-Zhang10}). We
apply it to the problem (\ref{Eqsm_new_4_4}) as follows.

By semidefinite relaxation (SDR) technique (see, e.g.,
\cite{Boyd-Vandenberghe04,Luo-Ma-So-Ye-Zhang10}), the problem
(\ref{Eqsm_new_4_4}) can be relaxed to
\begin{eqnarray}%
\label{Eqsm_new_4_5}\min_{\mX\in\mathbb{S}^{(NL+1)},~\mX\succeq
\textbf{0}}
&&\trace(\mC\mX)\\
\nonumber\mbox{subject to}~
~ &&\trace(\mE_p\mX)\unrhd~(4b_p-\textbf{ 1}'\diag(\va_p)\textbf{1}),~~~ p=1,\ldots,P,\\
\nonumber&&\trace(\mF_s\mX)=1, ~s=1,\ldots,(NL+1),
\end{eqnarray}
where $\mF_s$ is a matrix with $s$-row and $s$-column
$\mF_s(s,s)=1$, others are equal 0, for $s=1,\ldots,(NL+1)$,
\begin{eqnarray}
\nonumber\mE_p&\triangleq& \left(
                       \begin{array}{cc}
                         \diag(\va_p) &\textbf{ 1} \\
                         \textbf{1 }'& 0  \\
                       \end{array}
                     \right),\\
\nonumber\mC&\triangleq&\left(
                       \begin{array}{cc}
                         \mB &\mB\textbf{ 1} \\
                         \textbf{1 }'\mB&0 \\
                       \end{array}
                     \right),\\
\nonumber\mB&\triangleq&\diag( \omega_1\mB_{k+1},\ldots,
\omega_N\mB_{k+N}),
\end{eqnarray}
where  $\mB_{k+n}$ is defined by (\ref{Eqsm_new_4_3}) for
$n=1,\ldots,N$; $\mI$ and $\textbf{1}$ are an identity matrix and a
``1" vector with compatible dimensions respectively. The problem
(\ref{Eqsm_new_4_5}) is an SDP  problem. The derivation of the
problem (\ref{Eqsm_new_4_5})  is given in Appendix.


Note that the procedure for generating a feasible solution of the
problem (\ref{Eqsm_new_4_1}) from the solution of the problem
(\ref{Eqsm_new_4_5}) is also problem dependent, i.e., relying on the
equalities or inequalities (\ref{Eqsm_new_3_2}) and the Boolean
constraint. As an illustration, let us again consider the
representative constraints (\ref{Eqsm_20}), (\ref{Eqsm_21}) and
(\ref{Eqsm_new_3_10}).
Thus, the specific expressions of $\va_p$ and $b_p$ in the
optimization problem (\ref{Eqsm_new_4_5}) are given by
(\ref{Eqsm_new_3_12}) and (\ref{Eqsm_new_3_012}).

Based on the SDP (\ref{Eqsm_new_4_5}),
a typical Gaussian randomization procedure  is used to construct an
approximate solution to the problem (\ref{Eqsm_new_4_1}) here (see
\cite{Luo-Ma-So-Ye-Zhang10}). Thus, we have the following algorithm.

\begin{algorithm}[Sensor selection scheme with the energy constraint for correlated sensors]\label{alg_3}
~
\begin{itemize}
\item Step 1: Given an optimal solution of the SDP (\ref{Eqsm_new_4_5}) $X^*\in\mathbb{S}^{(NL+1)}$, and a number of randomizations  S.

\item Step 2: Generate $S$ feasible solutions by Gaussian randomization procedure based on $X^*$:

for s=1:S
\begin{itemize}
  \item[1.]
Generate a vector $\xi_s\sim \mathcal {N}(0, X^*)$. Set
$\eta_s\triangleq\xi_s(1:NL)$ which means the first NL entries of
$\xi_s$.

  \item[2.] Without loss of generality,  assume that $\omega_1\leq\omega_2\leq\ldots\leq\omega_N$. We
 generate the selection scheme from the $N$-th time step to the first time
 step. Set the index set of candidate sensors $\mathbbm{i}\triangleq\{1,\ldots,L\}$.

\begin{itemize}
%

\item Iteratively generate $\gamma_{k+N-n,s}^i$ for the $(N-n)$-th time step, $n=0,\ldots,(N-1)$

for~~ $n=0:(N-1)$
\begin{eqnarray}%
\label{Eqsm_new_4_7}\gamma_{k+N-n,s}^i &\triangleq&\left\{
                            \begin{array}{ll}
                              1, & \hbox{if~ $i\in \mathbbm{i}_1$,} \\
                              0, & \hbox{if~ $i\in \mathbbm{i}_2$,}
                            \end{array}
                          \right., for~ i=1,\ldots,L,
\end{eqnarray}
where $\mathbbm{i}_1$  is the index set of the first $m_{k+N-n}$
maximum entries of $\eta_s(J(L-n-1)+1:J(L-n))$ in the index set of
candidate sensors $\mathbbm{i}$ and
$\mathbbm{i}_2=\mathbbm{i}-\mathbbm{i}_1$. Set $m_k^i:=m_k^i-1$, for
$i\in\mathbbm{i}_1$. Update the index set of candidate sensors
$\mathbbm{i}\triangleq\{i:m_k^i>0, i=1,\ldots,L\}$.

end
\end{itemize}

\item[3.] Denote $\gamma_{k+N-n,s}\triangleq(\gamma_{k+N-n,s}^1,\ldots,\gamma_{k+N-n,s}^L)$
and $\gamma_{s}\triangleq(\gamma_{k+1,s},\ldots,\gamma_{k+N,s})$.
\end{itemize}
end

\item Step 3:  Determine $s^*=\mbox{argmax}_{s=1,\ldots,S}$
$f(\gamma_{s})$ where $f(\cdot)$ may be the objective funcrions
$f_1$, $f_2$ or $f_3$ defined in (\ref{Eqsm_8}), (\ref{Eqsm_9}) and
(\ref{Eqsm_new_3_4}) respectively.

\item Step 4:  Output  $\hat{\gamma}=\gamma_{s^*}$ as the sensor selections of  the problem (\ref{Eqsm_new_4_1}).

\end{itemize}
\end{algorithm}


Note that specific design of the randomization procedure technique
is problem dependent. Here, to construct the feasible solution
satisfying the constraints (\ref{Eqsm_20}), (\ref{Eqsm_21}) and
(\ref{Eqsm_new_3_10}), we employ Equation
(\ref{Eqsm_new_4_7}).
The choice of $S$ will  be discussed in  Section \ref{sec_6}.  Based
on simulations, the randomized solution can often achieve a good
performance with a small $S$, which is similar to that in
\cite{Luo-Ma-So-Ye-Zhang10}. The main computational complexity of
the algorithm is in Step 1 where an SDP problem needs to be solved.
The SDP problem can be solved efficiently by using interior-point
methods (see, e.g., \cite{Boyd-Vandenberghe04}).

\section{Numerical Examples}\label{sec_6}
In this section, we present a number of illustrative examples. Both
uncorrelated and correlated sensor measurement noise cases are
considered.
\subsection{Uncorrelated sensor measurement noises}\label{sec_6_1}
We first compare the performance of the approach given in Theorem
\ref{thm_3} with the one in Joshi an Boyd \cite{Joshi-Boyd09} and
the one in Mo et al. \cite{Mo-Ambrosino-Sinopoli11}.
\begin{example} Let us consider a dynamic system with $L=40$
sensors which are uniformly distributed over a square of size 100
$m$.
The parameter matrices and noise covariances for the dynamic system
(\ref{Eqsm_1})--(\ref{Eqsm_4}) are
\begin{eqnarray}%
\label{Eqsm_new_5_01}&&\mF_k= \left(
                                 \begin{array}{cccc}
                                   1 & 0\\
                                   0 & 1 \\
                                 \end{array}
                               \right),\mQ_k= \left(
                                 \begin{array}{cccc}
                                   5& 0 \\
                                   0 & 10\\
                                 \end{array}
                               \right),\\[3mm]
\label{Eqsm_new_5_001}&&~\mH_k^i=\left(
                                          \begin{array}{cccc}
                                            1 & 0 \\
                                            0 & 1 \\
                                          \end{array}
                                        \right),~\mR_k^i=\left(
                                          \begin{array}{cc}
                                            r_{1}^i & 0 \\
                                            0 & r_{2}^i \\
                                          \end{array}
                                        \right),
                                        i=1,\ldots,L,
\end{eqnarray}
where $r_{1}^i$ and $r_{2}^i$ are randomly sampled from the uniform
distribution in [5, 7] and [10 12] respectively.
We consider a constraint, i.e., select $m_{k+n}=[1, 5, 10, 15, 20]$
sensors from 40 sensors at the next time step respectively.
\end{example}
In Figure \ref{fig_01}, the traces of the estimation error
covariance  are plotted for the sensor selection method given in
Theorem \ref{thm_3}, the one in Joshi an Boyd \cite{Joshi-Boyd09}
and the one in Mo et al. \cite{Mo-Ambrosino-Sinopoli11}
respectively. The CPU time is plotted in Figure \ref{fig_02} for the
three algorithms respectively. Figure  \ref{fig_01} shows that the
three methods obtained very close and similar estimation performance
for the numerical example,  while Figure  \ref{fig_02} shows that
the CPU time of the method in Theorem \ref{thm_3} is much smaller
than that of the one in Joshi an Boyd \cite{Joshi-Boyd09} and the
one in Mo et al. \cite{Mo-Ambrosino-Sinopoli11}. The reason is that
the method in Theorem \ref{thm_3} is an analytical solution. In
addition, the computation time  of the three methods is not an
increasing function of the number of selected sensors. The reason is
that when the number of selected sensors increases, the number of
the decision variables does not increase and the structure of the
optimization does not change; only some parameters of the equality
constraints are changed.

Moreover, we  consider a representative target tracking dynamic
system with energy constraints. We assume that each target will be
tracked in a Cartesian frame. The four state variables include
position and velocity $(x, \dot{x}, y, \dot{y})$ respectively (see
e.g., \cite{BarShalom-Li95}). The parameter matrices and noise
covariances for the dynamic system (\ref{Eqsm_1})--(\ref{Eqsm_4})
are
\begin{eqnarray}%
\label{Eqsm_new_5_01}&&\mF_k= \left(
                                 \begin{array}{cccc}
                                   1 & T & 0 & 0 \\
                                   0 & 1& 0 &0 \\
                                   0 & 0 & 1 & T \\
                                   0 & 0 & 0& 1 \\
                                 \end{array}
                               \right),\mQ_k= \left(
                                 \begin{array}{cccc}
                                   T^3/3 & T^2/2 & 0 & 0 \\
                                   T^2/2 & T&0 &0 \\
                                   0 & 0 & T^3/3 & T^2/2\\
                                   0 & 0 & T^2/2& T \\
                                 \end{array}
                               \right),\\[3mm]
\label{Eqsm_new_5_001}&&~\mH_k^i=\left(
                                          \begin{array}{cccc}
                                            1 & 0 & 0 & 0 \\
                                            0 & 0& 1 & 0 \\
                                          \end{array}
                                        \right),~~~\mR_k^i,
                                        i=1,\ldots,L,
\end{eqnarray}
where $T=1$ s is the sampling interval; $\mF_k$, $\mQ_k$, $\mH_k^i$
are the same in the following examples. The difference is  the noise
covariance of measurements, $\mR_k^i$, in the following examples.
Since the algorithm in Joshi an Boyd \cite{Joshi-Boyd09} that
requires the measurement matrix is full-column rank when the
measurement matrices of each sensor are the same and the one in Mo
et al. \cite{Mo-Ambrosino-Sinopoli11} does not present how to
threshold the approximate solution to generate a feasible solution
satisfying the energy constraints,  we evaluate the performance of
Algorithm \ref{alg_2} by comparing with the exhaustive method for a
monitoring system which has a small number of sensors and using the
gap given in the Step 3 of Algorithm \ref{alg_2}  for a monitoring
system which has a large number of sensors in the following examples
respectively.

\begin{example} First, to compare with the exhaustive method, let us consider a relatively small monitoring system with $L=9$
sensors which are uniformly distributed over a square of size 100
$m$. The parameter matrices of the dynamic system are given in
(\ref{Eqsm_new_5_01})--(\ref{Eqsm_new_5_001}) where
\begin{eqnarray}
~\mR_k^i=\left(
                                          \begin{array}{cc}
                                            r_{1}^i & 0 \\
                                            0 & r_{2}^i \\
                                          \end{array}
                                        \right),
                                        \end{eqnarray}
$r_{1}^i$ and $r_{2}^i$ are randomly sampled from the uniform
distribution in [5, 10].
We consider the optimization problem (\ref{Eqsm_new_3_7}) with
temporally inseparable constraint (\ref{Eqsm_new_3_10}) and the
constraints (\ref{Eqsm_20}), (\ref{Eqsm_21}) where $N=3$,
$m_{k+n}=2$, $n=1,\ldots,N$ and $m_k^i=2$, i.e., select 2 sensors
from 9 sensors at each time step and select each sensor less than
twice in 3 time steps.
\end{example}

In Figure \ref{fig_1}, the traces of the \emph{final} estimation
error covariance $f_1$ based on the three methods are plotted
respectively, where $r_{1}^i$ and $r_{2}^i$ are randomly sampled 50
times. The three methods are 1) the exhaustive method that minimizes
the final estimation error covariance $f_1$, 2) Algorithm
\ref{alg_2} that maximizes the weighted information measure $f_3$
with weights $[1/3,~ 1/3,~ 1/3]$ and 3) Algorithm  \ref{alg_2} that
maximizes the weighted information measure $f_3$ with weights $[0,~
0,~ 1]$ respectively. Similarly, the traces of the \emph{average}
estimation error covariance $f_2$ are plotted in Figure \ref{fig_2}.
In Figure \ref{fig_3}, the sum of information measures of $N$ time
steps
$\sum_{n=1}^N\trace(\tilde{\mH}_{k+n}'\tilde{\mR}_{k+n}^{+}\tilde{\mH}_{k+n})$
is plotted  for the two sensor selection schemes respectively. They
are obtained from 2) and 3) respectively.


From Figures \ref{fig_1}--\ref{fig_3}, we have following
observations:
\begin{itemize}
\item Figure \ref{fig_1} shows that  the  trace of the \emph{final}
estimation error covariances obtained from Algorithm \ref{alg_2}
with two different weights that maximizes the weighted information
measure $f_3$ are very close to that of the exhaustive method.
Similarly, Figure \ref{fig_2} shows that  the trace of
\emph{average} estimation error covariance obtained from Algorithm
\ref{alg_2} with weights $[1/3,~ 1/3,~ 1/3]$ that maximizes the
weighted information measure $f_3$ is very close to that of the
exhaustive method. These indicate that maximization of the weighted
information measure $f_3$ is a good alternative criterion for
minimizing \emph{final} or \emph{average} estimation error
covariance for sensor selection.

\item Moreover, in Figure  \ref{fig_1}, when the objective is
minimization of the \emph{final} estimation error covariance $f_1$,
both Algorithm \ref{alg_2} with weights $[0,~ 0,~ 1]$ and Algorithm
\ref{alg_2} with weights $[1/3,~ 1/3,~ 1/3]$ are near optimal for
sensor selection. However, Figure  \ref{fig_3} shows that  the sum
of information measures of $N$ time steps for Algorithm \ref{alg_2}
with the weights $[1/3,~ 1/3,~ 1/3]$ is larger than that of
Algorithm \ref{alg_2} with the weights $[0,~ 0,~ 1]$. Thus, it is
better to choose the weights $[0,~ 0,~ 1]$, since a larger  sum of
information measures of $N$ time steps implies that more good
sensors are used.

\item Finally, Figure  \ref{fig_2} shows that when the objective is
to minimize the \emph{average} estimation error covariance $f_2$, it
is better to choose the weights $[1/3,~ 1/3,~ 1/3]$ since Algorithm
\ref{alg_2} with the weights $[1/3,~ 1/3,~ 1/3]$ is near optimal.

 \end{itemize}

\begin{example}
Next, let us consider a large monitoring system with $L=20\times
20=400$ sensors which are uniformly distributed in a square  of size
100 $m$. We consider the optimization problem (\ref{Eqsm_new_3_7})
with  temporally inseparable constraint (\ref{Eqsm_new_3_10}) and
the constraints (\ref{Eqsm_20}), (\ref{Eqsm_21}) where $N=5$,
$m_{k+n}=10$, $n=1,\ldots,N$ and $m_k^i=2$, i.e., select 10 sensors
at each time step from 400 sensors and select each sensor less than
twice in 5 time steps. Moreover, we consider the performance of
Algorithm \ref{alg_2} for different cases of $m_{k+n}$ from 10 to
100. Obviously, the exhaustive method is infeasible.
\end{example}

In Figure \ref{fig_4}, the upper bound and lower bound of the
objective function of the optimization problem (\ref{Eqsm_new_3_7})
are plotted based on 50 Monte Carlo runs. The corresponding gaps,
i.e, the upper bound minus the lower bound are plotted in Figure
\ref{fig_5}. Figures \ref{fig_4} and \ref{fig_5} show that the gaps
are very small and Algorithm \ref{alg_2} can obtain the optimal
solution in the sense of maximizing the weighted information measure
$f_3$ in many cases, although the sensor network is large where  the
Boolean decision variables are more than 2000. Figure \ref{fig_05}
presents the gaps as a function of $m_{k+n}$ from 10 to 100. It
shows that the gaps are increasing as the number of selected
sensors.


\subsection{Correlated sensor measurement noises}\label{sec_6_2}
In this subsection, we will compare Algorithm \ref{alg_3} with the
exhaustive method for a simple problem so that the approximation
loss can be computed. For this, we assume that only the sensor
selection scheme for the next step is to be designed, i.e., $N=1$.
For $N=5$, we will compare Algorithm \ref{alg_3} with Algorithm
\ref{alg_2} that ignores dependence. In this case, the exhaustive
method is infeasible, since we have to enumerate $2.43\times
10^{10}$ cases. At the end, an example that compares the root mean
square error (RMSE) of state estimation based on sensor selection is
presented.

\begin{example}
Let us consider $L=25$ for the sensor network shown in Figure
\ref{fig_0}. Assume that there is a jammer signal $\vv_k^0$ with a
covariance $\mR_k^0$ at the position $(550~m, 200~ m)$, besides the
natural noises $\vv_k^i, i=1,\ldots,L$ which are independent of
$\vv_k^0$. The jamming signal introduces dependence among
measurement noises. Thus, the noises at the $i$-th sensor is given
as follows
\begin{eqnarray}
\label{Eqsm_new_5_c01}\breve{\vv}_k^i &=&\vv_k^i+\frac{P_0}{1+\alpha
d_{i,0}^n}\vv_k^0=\vv_k^i+\beta_i\vv_k^0
\end{eqnarray}
where $\beta_i\triangleq\frac{P_0}{1+\alpha d_{i,0}^n}$; $d_{i,0}$
is the distance between the jammer and the $i$-th sensor; the signal
decay exponent $n=2$, the scaling parameter $\alpha=1$ and different
values for the signal power
$P_0=[1,~3,~6,~10,~12,~15,~20]\times10^5$ are used in simulations
respectively. Thus, noises of sensors are correlated and the $i$-th
block and $j$-th block of the noise covariance $\breve{\mR}_k$ can
be computed by (\ref{Eqsm_new_5_c01}) to be
\begin{eqnarray}
\label{Eqsm_new_5_c1}\breve{\mR}_k^{ij}&=&Cov(\breve{\vv}_k^i,\breve{\vv}_k^j)
=\mR_k^{ij}+\beta_i\beta_j\mR_k^0,
\end{eqnarray}
where
\begin{eqnarray}
\nonumber \mR_k^{ij}=\left\{
              \begin{array}{ll}
                \left(
           \begin{array}{cc}
             10 & 0\\
             0 & 10  \\
           \end{array}
         \right), & \hbox{if $i=j$} \\[5mm]
                \left(
           \begin{array}{cc}
             0 & 0\\
             0 & 0 \\
           \end{array}
         \right), & \hbox{$i\neq j$.}
              \end{array}
            \right.
,~~ \mR_k^0=\left(
           \begin{array}{cc}
             1 & 0\\
             0 & 1\\
           \end{array}
         \right),
\end{eqnarray}
are used in simulations. Note that the corresponding Pearson's
correlation coefficients between sensors are approximately equal
 to $[0.1,~0.2,~0.4,$ $~0.6,~0.7$, $~0.8,~0.9]$ corresponding
to $P_0=[1,~3,~6,~10,~12,~15,~20]\times10^5$ respectively. We
consider the optimization problem (\ref{Eqsm_new_4_1}) with
temporally separable constraints (\ref{Eqsm_20}), (\ref{Eqsm_21})
where $N=1$, $m_{k+1}=2$,   i.e., select 2 sensors from 25 sensors
at the next time step.

\end{example}

In Figure \ref{fig_6}, comparisons of the objective function
$\trace(\tilde{\mH}_{k+1}'\tilde{\mR}_{k+1}^{+}\tilde{\mH}_{k+1})$
of the optimization problem (\ref{Eqsm_new_4_1}) (i.e., the
information measure of the $(k+1)$-th time step) based on the
exhaustive method, Algorithm \ref{alg_3} and Theorem \ref{thm_3}
that ignores dependence are plotted for different jammer signal
powers respectively. We present the performance of Algorithm
\ref{alg_3} with small numbers of randomizations $S=20, S=100$.
Similarly, comparisons of the traces of the estimation error
covariance of $(k+1)$-th time step are plotted in Figure
\ref{fig_7}.

From Figures \ref{fig_6}--\ref{fig_7}, we have the following
observations:
\begin{itemize}
\item For all the three
methods, the larger is the signal power of jammer, the smaller is
the information measure of the $(k+1)$-th time step obtained from
the selected sensors and the larger is the trace of the  estimation
error covariance at the $(k+1)$-th time step.

\item Figures \ref{fig_6}--\ref{fig_7} also show that the exhaustive
method yields better results than Algorithm \ref{alg_3}  with small
$S$ and the latter is better than the method of Theorem \ref{thm_3}
that ignores dependence, especially in the case of strong dependence
(i.e., strong signal power of the jammer).

\item Figures \ref{fig_6}--\ref{fig_7} indicate that larger the value of $S$ is , the closer is the performance of Algorithm \ref{alg_3} to
that of the exhaustive method, i.e., the smaller is the
approximation loss of Algorithm \ref{alg_3}.
\end{itemize}

\begin{example}
Next, let us consider a  monitoring system with a large $N=5$ and
$L=5\times 5=25$ sensors which are uniformly distributed in a square
of size 100 m . We consider the optimization problem
(\ref{Eqsm_new_3_7}) with temporally inseparable constraint
(\ref{Eqsm_new_3_10}) and the constraints (\ref{Eqsm_20}),
(\ref{Eqsm_21}) where $m_{k+n}=2$, $n=1,\ldots,N$ and $m_k^i=2$,
i.e., select 2 sensors at each time step from 25 sensors and select
each sensor less than twice in next 5 time steps.
\end{example}

In Figure \ref{fig_8}, comparisons of the objective function
$\sum_{n=1}^N\omega_n\trace(\tilde{\mH}_{k+n}'\tilde{\mR}_{k+n}^{+}\tilde{\mH}_{k+n})$
of the optimization problem (\ref{Eqsm_new_4_1}) (i.e., the sum of
the weighted information measure of $N$ time steps, $f_3$ defined in
(\ref{Eqsm_new_3_4})) based on approaches of Algorithm \ref{alg_3}
and Algorithm \ref{alg_2} that ignores dependence are plotted for
different jammer signal powers respectively. We examine the
performance of Algorithm \ref{alg_3} as a function of  the number of
randomizations $S=20, S=100, S=2000$ which are small, compared with
the exhaustive number $2.43\times 10^{10}$. Similarly, comparisons
of the traces of the average estimation error covariances of $N$
time steps are plotted in Figure \ref{fig_9}.

 Figures \ref{fig_8}--\ref{fig_9}  show that Algorithm
\ref{alg_3} with a small value of $S$ is better than Algorithm
\ref{alg_2} that ignores dependence, especially in the case of
strong dependence (i.e., strong signal power of the jammer). In
addition, Figures \ref{fig_8}--\ref{fig_9} indicate that  larger the
value of $S$ is, the better is the performance of Algorithm
\ref{alg_3} than that of Algorithm \ref{alg_2} that ignores
dependence.

\begin{example}
Finally, let us consider the $L$-sensor noise covariance $\mR_k,
k=1,2,\ldots$  which  depends on the state $\vx_k$. A frequently
made assumption is that larger is the distance between the sensor
and the target,  larger is the noise covariance. However, when we
design the sensor selection scheme of next $N$ time steps at time
$k$, we do not know the state  $\vx_{k+n}$ so that we replace it by
the state prediction $\vx_{k+n|k}$ which is used to compute the
$\mR_{k+n}, n=1,\ldots, N$. Specifically, the noise covariance
$\mR_{k+n}$ is
\begin{eqnarray}%
           \mR_{k+n} = \bar{\mR}_{k+n}+ \breve{\mR}_{k+n},
\end{eqnarray}
where $\breve{\mR}_{k+n}$ is the noise covariance from the jammer
signal $\vv_k^0$ defined in (\ref{Eqsm_new_5_c1}) and the signal
power of jammer $P_0=1.5\times 10^6$; $\bar{\mR}_{k+n}$ is a
diagonal matrix with the $i$-th diagonal block defined as follows
\begin{eqnarray}%
           \bar{\mR}_{k+n}^{ii}=\left(
                           \begin{array}{cc}
                              \alpha_1 d_{i,n}& 0 \\
                              0& \alpha_1 d_{i,n} \\
                           \end{array}
            \right),~ i=1,\ldots,L,~, n=1,\ldots,N,
\end{eqnarray}
where  $\alpha_1=0.05$ is a scaling parameter; $d_{i,n}$ is the
distance between the target prediction $\vx_{k+n|k}$ and the $i$-th
sensor. We consider the optimization problem (\ref{Eqsm_new_3_7})
with temporally separable constraint (\ref{Eqsm_new_3_10}) and the
constraints (\ref{Eqsm_20}), (\ref{Eqsm_21}) where $N=5$,
$m_{k+n}=2$, $n=1,\ldots,N$ and $m_k^i=2$, i.e., select 2 sensors at
each time step and select each sensor less than twice in 5 time
steps. The initial state of the target is $(600~m, -20~m/s, 200~ m,
0~m/s)$.
\end{example}
In Figure  \ref{fig_10}, RMSE of the state estimates based on 200
Monte Carlo runs is given. We compare   Algorithm \ref{alg_2} that
ignores dependence with Algorithm \ref{alg_3} with S=20, S=100 and
S=2000 respectively.

As far as the RMSE is considered, Figure  \ref{fig_10} also shows
that Algorithm \ref{alg_3} with a small value of $S$ is better than
Algorithm \ref{alg_2} that ignores dependence and that larger the
value of $S$ is, the better is the performance of Algorithm
\ref{alg_3} than that of Algorithm \ref{alg_2} that ignores
dependence.

\section{Conclusion}\label{sec_7}
In this paper,  we have proposed a generalized information filter
for target tracking in wireless sensor networks where measurements
from a subset of sensors are employed at each time step. Then, under
a regularity condition, we proved that the multistage look-ahead
policy that minimizes either the final or the average estimation
error covariances of next $N$ time steps is equivalent to the myopic
sensor selection policy that maximizes the trace of the generalized
information gain at each time step. When the measurement noises are
uncorrelated, the optimal solution has been derived analytically for
sensor selection with temporally separable constraints. For
temporally inseparable constraints, the sensor selection scheme can
be obtained by approximately solving an LP problem. Although there
is no guarantee that the gap between the performance of the chosen
subset and the performance bound is always small, numerical examples
showed that the algorithm is near-optimal in many cases and   the
selection scheme for a large sensor network with more than 2000
Boolean decision variables
 can be dealt with quickly. Finally, when the noises
of measurements are correlated, the sensor selection problem with
temporally inseparable constraints was relaxed to a BQP problem
which can be efficiently solved by a Gaussian randomization
procedure by solving an SDP problem which can be solved by
interior-point methods and related software tools. Numerical
examples showed that the proposed method is much better than the
method that ignores dependence.

Future work will involve the generalization from the linear dynamic
systems to nonlinear dynamic systems. The equivalence  between
multistage look-ahead optimization policy for sensor management and
the myopic sensor optimization policy and the corresponding sensor
management schemes will be investigated. In addition, it can be
considered for various applications such as robotics, sensor
placement for structures and different types of wireless networks.

\section*{Appendix}
\textbf{The proof of Theorem \ref{thm_1}.}
\begin{proof} Notice that there are $m_{k+1}$ number of $\gamma_{k+1}^i=1$ and $L-m_{k+1}$
number of $\gamma_{k+1}^i=0$ so that there exists a  permutation
matrix $\mP$ such that
\begin{eqnarray}
\label{Eqsm_A_1}\mP\tilde{\mH}_{k+1} &=&\left(
                         \begin{array}{cc}
                           \mH_{k+1}(1:m_{k+1})  \\
                           \textbf{0}   \\
                         \end{array}
                       \right)
\end{eqnarray}
and
\begin{eqnarray}
\label{Eqsm_A_2}\mP\tilde{\mR}_{k+1}\mP' &=&\left(
                         \begin{array}{cc}
                           \mR_{k+1}(1:m_{k+1}) & \textbf{0} \\
                           \textbf{0} & \textbf{0} \\
                         \end{array}
                       \right),
\end{eqnarray}
where  $\mH_{k+1}(1:m_{k+1})$ and $\mR_{k+1}(1:m_{k+1})$  are the
stacked measurement matrices and the covariance of noises  of the
$m_{k+1}$ selected sensor respectively; $\textbf{0}$ is a zero
matrix with compatible dimensions. From the property of the
permutation matrix $\mP\times\mP'=\mP'\times\mP=\mI$ and the
definition of Moore-Penrose inverse, we have
\begin{eqnarray}
\nonumber\mP\tilde{\mR}_{k+1}^{+} \mP'&=&(\mP\tilde{\mR}_{k+1}\mP')^{+} \\[3mm]
\nonumber   &=&\left(
                         \begin{array}{cc}
                           \mR_{k+1}(1:m_{k+1}) & \textbf{0} \\
                           \textbf{0} & \textbf{0} \\
                         \end{array}
                       \right)^{+}\\[3mm]
\label{Eqsm_A_3}&=&\left(
                         \begin{array}{cc}
                           (\mR_{k+1}(1:m_{k+1}))^{-1} & \textbf{0} \\
                           \textbf{0} & \textbf{0} \\
                         \end{array}
                       \right)
\end{eqnarray}
Moreover, by equations (\ref{Eqsm_A_1})--(\ref{Eqsm_A_3}) and
repeatedly using the definition of Moore-Penrose generalized inverse
and the property of the permutation matrix
$\mP\times\mP'=\mP'\times\mP=\mI$, we have the following derivation
\begin{eqnarray}
\nonumber&&\left(
\tilde{\mR}_{k+1}+\tilde{\mH}_{k+1}\mP_{k+1|k}\tilde{\mH}_{k+1}'
                     \right)^{+}\\[3mm]
\nonumber&=&\left( \mP'\left(\begin{array}{cc}
            \mR_{k+1}(1:m_{k+1}) &\textbf{0} \\
                                            \textbf{0} & \textbf{0} \\
                                          \end{array}\right)\mP\right.\\
\nonumber&&\left. +\mP '\left(\begin{array}{cc}
             \tilde{\mH}_{k+1}(1:m_{k+1})\mP_{k+1|k}\tilde{\mH}_{k+1}(1:m_{k+1})') &\textbf{0} \\
                                            \textbf{0} & \textbf{0} \\
                                          \end{array}\right)\mP
                     \right)^{+}\\[3mm]
\nonumber&=&\left( \mP'\left(\begin{array}{cc}
            \mR_{k+1}(1:m_{k+1})+ \tilde{\mH}_{k+1}(1:m_{k+1})\mP_{k+1|k}\tilde{\mH}_{k+1}(1:m_{k+1})') &\textbf{0} \\
                                            \textbf{0} & \textbf{0} \\
                                          \end{array}\right)\mP
                     \right)^{+}
\end{eqnarray}
\begin{eqnarray}
\nonumber&=&\mP'\left(\begin{array}{cc}
              (\mR_{k+1}(1:m_{k+1})+ \tilde{\mH}_{k+1}(1:m_{k+1})\mP_{k+1|k}\tilde{\mH}_{k+1}(1:m_{k+1})') &\textbf{0} \\
                                            \textbf{0} & \textbf{0} \\
                                          \end{array}
                      \right)^{+}\mP\\[3mm]
\nonumber&=&\mP'\left(\begin{array}{cc}
              (\mR_{k+1}(1:m_{k+1})+ \tilde{\mH}_{k+1}(1:m_{k+1})\mP_{k+1|k}\tilde{\mH}_{k+1}(1:m_{k+1})')^{-1} &\textbf{0} \\
                                            \textbf{0} & \textbf{0} \\
                                          \end{array}
                      \right)\mP\\[3mm]
\nonumber&=&\mP'\left(
                             \begin{array}{cc}
                                          (\mR_{k+1}(1:m_{k+1}))^{-1}  (\mI+ \tilde{\mH}_{k+1}(1:m_{k+1}) &\textbf{0} \\
                                           \cdot\mP_{k+1|k}\tilde{\mH}_{k+1}(1:m_{k+1})'(\mR_{k+1}(1:m_{k+1}))^{-1})^{-1} & \\[3mm]
                                            \textbf{0} & \textbf{0} \\
                              \end{array}
                      \right)\mP\\[3mm]
\nonumber&=&\mP'\left(
                         \begin{array}{cc}
                           (\mR_{k+1}(1:m_{k+1}))^{-1} & \textbf{0} \\
                           \textbf{0} & \textbf{0} \\
                         \end{array}
                       \right)\\[3mm]
\nonumber&&\cdot\left(
                                          \begin{array}{cc}
                                           (\mI+ \tilde{\mH}_{k+1}(1:m_{k+1})\mP_{k+1|k}\tilde{\mH}_{k+1}(1:m_{k+1})'(\mR_{k+1}(1:m_{k+1}))^{-1})^{-1} &\textbf{0} \\
                                            \textbf{0} & \mI \\
                                          \end{array}
                      \right)\mP
\end{eqnarray}
\begin{eqnarray}
\nonumber&=&\mP'\left(
                         \begin{array}{cc}
                           (\mR_{k+1}(1:m_{k+1}))^{-1} & \textbf{0} \\
                           \textbf{0} & \textbf{0} \\
                         \end{array}
                       \right)\\[3mm]
\nonumber&&\cdot\left(\mI+\left(
                                          \begin{array}{cc}
                                            \tilde{\mH}_{k+1}(1:m_{k+1})\mP_{k+1|k}\tilde{\mH}_{k+1}(1:m_{k+1})' &\textbf{0} \\
                                            \textbf{0} & \textbf{0} \\
                                          \end{array}
                                        \right)\right.\\[3mm]
\nonumber&&\cdot\left.\left(
                         \begin{array}{cc}
                           (\mR_{k+1}(1:m_{k+1}))^{-1} & \textbf{0} \\
                           \textbf{0} & \textbf{0} \\
                         \end{array}
                       \right)\right)^{-1}\mP\\[3mm]
\nonumber&=&\mP'\left(
                         \begin{array}{cc}
                           (\mR_{k+1}(1:m_{k+1}))^{-1} & \textbf{0} \\
                           \textbf{0} & \textbf{0} \\
                         \end{array}
                       \right)\\[3mm]
\nonumber&&\cdot\left(\mP' +\mP' \left(
                                          \begin{array}{cc}
                                            \tilde{\mH}_{k+1}(1:m_{k+1})\mP_{k+1|k}\tilde{\mH}_{k+1}(1:m_{k+1})' &\textbf{0} \\
                                            \textbf{0} & \textbf{0} \\
                                          \end{array}
                                        \right)\right.\\[3mm]
\nonumber&&\cdot\left.\left(
                         \begin{array}{cc}
                           (\mR_{k+1}(1:m_{k+1}))^{-1} & \textbf{0} \\
                           \textbf{0} & \textbf{0} \\
                         \end{array}
                       \right)\right)^{-1}\\[3mm]
\nonumber&=&\mP'\left(
                         \begin{array}{cc}
                           (\mR_{k+1}(1:m_{k+1}))^{-1} & \textbf{0} \\
                           \textbf{0} & \textbf{0} \\
                         \end{array}
                       \right)\\
\nonumber&&\left(\mP'+\mP'\left(
                                          \begin{array}{cc}
                                            \tilde{\mH}_{k+1}(1:m_{k+1})\mP_{k+1|k}\tilde{\mH}_{k+1}(1:m_{k+1})' &\textbf{0} \\
                                            \textbf{0} & \textbf{0} \\
                                          \end{array}
                                        \right)
                       \mP\tilde{\mR}_{k+1}^{+}\mP'\right)^{-1}\\[3mm]
\nonumber&=&\mP'\left(
                         \begin{array}{cc}
                           (\mR_{k+1}(1:m_{k+1}))^{-1} & \textbf{0} \\
                           \textbf{0} & \textbf{0} \\
                         \end{array}
                       \right)(\mP'+\mP'\mP\tilde{\mH}_{k+1}\mP_{k+1|k}\tilde{\mH}_{k+1}'\mP'\mP\tilde{\mR}_{k+1}^{+}\mP')^{-1}\\[3mm]
\nonumber&=&\mP'\left(
                         \begin{array}{cc}
                           (\mR_{k+1}(1:m_{k+1}))^{-1} & \textbf{0} \\
                           \textbf{0} & \textbf{0} \\
                         \end{array}
                       \right)(\mP'+\tilde{\mH}_{k+1}\mP_{k+1|k}\tilde{\mH}_{k+1}'\tilde{\mR}_{k+1}^{+}\mP')^{-1}\\[3mm]
\nonumber&=&\mP'\left(
                         \begin{array}{cc}
                           (\mR_{k+1}(1:m_{k+1}))^{-1} & \textbf{0} \\
                           \textbf{0} & \textbf{0} \\
                         \end{array}
                       \right)\mP(\mI+\tilde{\mH}_{k+1}\mP_{k+1|k}\tilde{\mH}_{k+1}'\tilde{\mR}_{k+1}^{+})^{-1}\\[3mm]
\label{Eqsm_A_4}&=&\tilde{\mR}_{k+1}^{+}(\mI+\tilde{\mH}_{k+1}\mP_{k+1|k}\tilde{\mH}_{k+1}'\tilde{\mR}_{k+1}^{+})^{-1}.
\end{eqnarray}

From Equations (\ref{Eqsm_13}) and (\ref{Eqsm_A_4}), we have
\begin{eqnarray}
\nonumber\mK_{k+1}&=&\mP_{k+1|k}\mH_{k+1}'\left(
\tilde{\mR}_{k+1}+\tilde{\mH}_{k+1}\mP_{k+1|k}\tilde{\mH}_{k+1}'
                     \right)^{+}\\
\nonumber&=&\mP_{k+1|k}\mH_{k+1}'\tilde{\mR}_{k+1}^{+}\left(
\mI+\tilde{\mH}_{k+1}\mP_{k+1|k}\tilde{\mH}_{k+1}'\tilde{\mR}_{k+1}^{+}
                     \right)^{-1},
\end{eqnarray}
which yields
\begin{eqnarray}
\nonumber\mK_{k+1}\left(
\mI+\tilde{\mH}_{k+1}\mP_{k+1|k}\tilde{\mH}_{k+1}'\tilde{\mR}_{k+1}^{+}
                     \right)
\nonumber&=&\mP_{k+1|k}\mH_{k+1}'\tilde{\mR}_{k+1}^{+},
\end{eqnarray}
so that
\begin{eqnarray}
\nonumber\mK_{k+1}
\mI+\mK_{k+1}\tilde{\mH}_{k+1}\mP_{k+1|k}\tilde{\mH}_{k+1}'\tilde{\mR}_{k+1}^{+}
\nonumber&=&\mP_{k+1|k}\mH_{k+1}'\tilde{\mR}_{k+1}^{+}.
\end{eqnarray}
Moreover, we have
\begin{eqnarray}
\nonumber\mK_{k+1}
&=&(\mI-\mK_{k+1}\tilde{\mH}_{k+1})\mP_{k+1|k}\mH_{k+1}'\tilde{\mR}_{k+1}^{+}\\
\label{Eqsm_A_5}&=&\mP_{k+1|k+1}\mH_{k+1}'\tilde{\mR}_{k+1}^{+},
\end{eqnarray}
where the derivation of (\ref{Eqsm_A_5}) is based on
(\ref{Eqsm_11}).

From Equations (\ref{Eqsm_11}) and (\ref{Eqsm_A_5}), we have
\begin{eqnarray}
\nonumber\mP_{k+1|k}^{-1}&=&\mP_{k+1|k+1}^{-1}(\mI-\mK_{k+1}\tilde{\mH}_{k+1}),\\
\nonumber&=&\mP_{k+1|k+1}^{-1} -\mP_{k+1|k+1}^{-1} \mK_{k+1}\tilde{\mH}_{k+1},\\
\nonumber&=&\mP_{k+1|k+1}^{-1}
-\mH_{k+1}'\tilde{\mR}_{k+1}^{+}\tilde{\mH}_{k+1}.
\end{eqnarray}
Thus, we have
$\mP_{k+1|k+1}^{-1}=\mP_{k+1|k}^{-1}+\mH_{k+1}'\tilde{\mR}_{k+1}^{+}\tilde{\mH}_{k+1}$.

By (\ref{Eqsm_10}), (\ref{Eqsm_11}) and (\ref{Eqsm_A_5}),
\begin{eqnarray}%
\nonumber\vx_{k+1|k+1}&=&(\mI-\mK_{k+1}\tilde{\mH}_{k+1})\vx_{k+1|k}+\mK_{k+1}\vz_{k+1},\\
\nonumber&=&(\mI-\mK_{k+1}\tilde{\mH}_{k+1})\vx_{k+1|k}+\mP_{k+1|k+1}\mH_{k+1}'\tilde{\mR}_{k+1}^{+}\vz_{k+1},\\
\nonumber&=&\mP_{k+1|k+1}\mP_{k+1|k}^{-1}\vx_{k+1|k}+\mP_{k+1|k+1}\mH_{k+1}'\tilde{\mR}_{k+1}^{+}\vz_{k+1},\\
\nonumber&=&\mP_{k+1|k+1}\{\mP_{k+1|k}^{-1}\vx_{k+1|k}+\mH_{k+1}'\tilde{\mR}_{k+1}^{+}\vz_{k+1}\},
\end{eqnarray}
Thus, we have $\vx_{k+1|k+1}=\mP_{k+1|k+1}
\{\mP_{k+1|k}^{-1}\vx_{k+1|k}+
\tilde{\mH}_{k+1}'\tilde{\mR}_{k+1}^{+}\vz_{k+1}\}$.
 \Box
\end{proof}

\textbf{The proof of Lemma \ref{lem_1}.}
\begin{proof}
First, we prove that the optimal solution of ($A_1$) is also the
optimal solution of ($A_2$). If $\vx_1$ is the optimal solution of
($A_1$), then, for arbitrary $\vx\in \mathcal {S}$,
$\mM(\vx)\preceq\mM(\vx_1)$ which yields
$\trace(\mM(\vx))\leq\trace(\mM(\vx_1))$. Thus, $\vx_1$ is also the
optimal solution of ($A_2$).

On the other hand, if $\vx_2$ is the optimal solution of ($A_2$),
then, for arbitrary $\vx\in \mathcal {S}$,
$\trace(\mM(\vx))\leq\trace(\mM(\vx_2))$  which implies
$\trace(\mM(\vx_1))\leq\trace(\mM(\vx_2))$. Notice that the problem
($A_1$) has an optimal solution $\vx_1$ which yields
$\trace(\mM(\vx_2))\leq\trace(\mM(\vx_1))$. Thus,
$\trace(\mM(\vx_1))=\trace(\mM(\vx_2))$ so that
$\trace(\mM(\vx_1)-\mM(\vx_2))=0$. By
$\trace(\mM(\vx_1)-\mM(\vx_2))=0$ and
$\mM(\vx_1)-\mM(\vx_2)\succeq0$, we have $\mM(\vx_1)=\mM(\vx_2)$.
Therefore, $\vx_2$ is also the optimal solution of ($A_1$).\Box
\end{proof}

\textbf{The proof of Lemma \ref{lem_2}.}
\begin{proof}
In one direction: if $\vx_1^*,\ldots,\vx_N^*$ is the optimal
solution of (($B_1$) with $n=N$), then, for $n=1,\ldots,N$,
\begin{eqnarray}%
\nonumber \mM_n(\vx_1^*,\ldots,\vx_n^*)\preceq
\mM_n(\vx_1,\ldots,\vx_n) ~\mbox{for arbitrary}~ \vx_i\in\mathcal
{S}_i, i=1,\ldots,n,
\end{eqnarray}
since the optimal solution that minimizes
$\mM_n(\vx_1,\ldots,\vx_n)$ is the same as that minimizes
$\mM_{n+1}(\vx_1,\ldots,\vx_{n+1})$, for $n=1,\ldots,N$. Thus,
\begin{eqnarray}%
\nonumber\sum_{n=1}^N\mM_n(\vx_1^*,\ldots,\vx_n^*)\preceq
\sum_{n=1}^N\mM_n(\vx_1,\ldots,\vx_n) ~\mbox{for arbitrary}~
\vx_i\in\mathcal {S}_i, i=1,\ldots,n,
\end{eqnarray}
which yields $\vx_1^*,\ldots,\vx_N^*$ is also the optimal solution
of ($B_2$).

On the other hand,  assume that $\vx_1^*,\ldots,\vx_N^*$ is the
optimal solution of ($B_2$). If $\vx_1^*,\ldots,\vx_N^*$ is not the
optimal solution of ($B_1$), then there exists an optimal solution
$\vx_1^0,\ldots,\vx_N^0$ which has a smaller objective function
value than that of $\vx_1^*,\ldots,\vx_N^*$. Since the optimal
solution that minimizes $\mM_n(\vx_1,\ldots,\vx_n)$ is the same as
that minimizes $\mM_{n+1}(\vx_1,\ldots,\vx_{n+1})$,
$\vx_1^0,\ldots,\vx_n^0$ is the optimal solution of
$\mM_{n}(\vx_1,\ldots,\vx_{n})$ for $n=1,\ldots,N$. Thus,
\begin{eqnarray}%
\nonumber \mM_n(\vx_1^*,\ldots,\vx_n^*)\succeq
\mM_n(\vx_1^0,\ldots,\vx_n^0), n=1,\ldots,N,
\end{eqnarray}
so that
\begin{eqnarray}%
\nonumber\sum_{n=1}^N\mM_n(\vx_1^*,\ldots,\vx_n^*)\succeq
\sum_{n=1}^N\mM_n(\vx_1^0,\ldots,\vx_n^0),
\end{eqnarray}
which yields a contradiction. Thus,  $\vx_1^*,\ldots,\vx_N^*$ is the
optimal solution of ($B_1$).\Box
\end{proof}

\textbf{The proof of Lemma \ref{lem_3}.}
\begin{proof}
If the problem (\ref{Eqsm_19}) has an optimal solution, from the
fact that any positive definite matrix $A\succeq A_0$ implies
$A^{-1}\preceq A_0^{-1}$, then we have that the problem
(\ref{Eqsm_19}) is equivalent to solve
\begin{eqnarray}%
\nonumber \max_{\gamma_{k+n}^i} ~&&(\mP_{k+N|k+N}(\gamma_{k+1}^1,\ldots,\gamma_{k+1}^L,\ldots,\gamma_{k+N}^1,\ldots,\gamma_{k+N}^L))^{-1}\\
\nonumber&&=\mP_{k+N|k+N-1}^{-1}+\tilde{\mH}_{k+N}'\tilde{\mR}_{k+N}^{+}\tilde{\mH}_{k+N}\\
\nonumber\mbox{subject to}~
~ &&\sum_{i=1}^L\gamma_{k+n}^i=m_{k+n},~ n=1,\ldots,N,\\
\nonumber&&\gamma_{k+n}^i\in \{0,1\}, i=1,2,\ldots,L,~ n=1,\ldots,N,
\end{eqnarray}
which has the same optimal solution.  Since the constraints are
temporally separable, and by Equation (\ref{Eqsm_14}) and the
invertibility of $\mF_k$, it is equivalent to solve the following
two problems
\begin{eqnarray}%
\nonumber \max_{\gamma_{k+n}^i} ~&& \tilde{\mH}_{k+N}'\tilde{\mR}_{k+N}^{+}\tilde{\mH}_{k+N}\\
\mbox{subject to}~
\nonumber &&\sum_{i=1}^L\gamma_{k+n}^i=m_{k+n},~ n= N,\\
\nonumber&&\gamma_{k+n}^i\in \{0,1\}, i=1,2,\ldots,L,~ n= N,
\end{eqnarray}
and
\begin{eqnarray}%
\nonumber \min_{\gamma_{k+n}^i} ~&&\mP_{k+N-1|k+N-1}(\gamma_{k+1}^1,\ldots,\gamma_{k+1}^L,\ldots,\gamma_{k+N-1}^1,\ldots,\gamma_{k+N-1}^L)\\
\nonumber&&=(\mP_{k+N-1|k+N-2}^{-1}+\tilde{\mH}_{k+N-1}'\tilde{\mR}_{k+N-1}^{+}\tilde{\mH}_{k+N-1})^{-1}\\
\nonumber\mbox{subject to}~
~ &&\sum_{i=1}^L\gamma_{k+n}^i=m_{k+n},~ n=1,\ldots, N,-1,\\
\nonumber&&\gamma_{k+n}^i\in \{0,1\}, i=1,2,\ldots,L,~ n=1,\ldots,
N,-1.
\end{eqnarray}
Both of them have an optimal solution respectively.

After $N$-step recursive decomposition, the problem (\ref{Eqsm_19})
is equivalent to solve the following $N$ optimization problems
\begin{eqnarray}%
\nonumber\max_{\gamma_{k+n}^i} ~&&\tilde{\mH}_{k+n}'\tilde{\mR}_{k+n}^{+}\tilde{\mH}_{k+n}~~~~~ \mbox{for}~ n=1,\ldots, N,\\
\nonumber\mbox{subject to}~
~ &&\sum_{i=1}^L\gamma_{k+n}^i=m_{k+n},\\
\nonumber&&\gamma_{k+n}^i\in \{0,1\}, i=1,2,\ldots,L.
\end{eqnarray}
All of $N$ optimization problems have an optimal solution
respectively.


Moreover, for the problem (\ref{Eqsm_19}), we consider minimizing
$\mP_{k+n|k+n}$ and $\mP_{k+n+1|k+n+1}$ respectively. Both of them
have a recursive decomposition similar to that of minimizing
$\mP_{k+N|k+N}$ for the problem (\ref{Eqsm_19}). Thus, we have
\begin{eqnarray}%
 \min_{\gamma_{k+s}^i} ~&&\mP_{k+n|k+n}(\gamma_{k+1}^1,\ldots,\gamma_{k+1}^L,\ldots,\gamma_{k+n}^1,\ldots,\gamma_{k+n}^L)\\
\nonumber\mbox{subject to}~
~ &&\sum_{i=1}^L\gamma_{k+s}^i=m_{k+s},~ s=1,\ldots,n,\\
\nonumber&&\gamma_{k+s}^i\in \{0,1\}, i=1,2,\ldots,L,~ s=1,\ldots,n,
\end{eqnarray}
and
\begin{eqnarray}%
 \min_{\gamma_{k+s}^i} ~&&\mP_{k+n+1|k+n+1}(\gamma_{k+1}^1,\ldots,\gamma_{k+1}^L,\ldots,\gamma_{k+n+1}^1,\ldots,\gamma_{k+n+1}^L)\\
\nonumber\mbox{subject to}~
~ &&\sum_{i=1}^L\gamma_{k+s}^i=m_{k+s},~ s=1,\ldots,n+1,\\
\nonumber&&\gamma_{k+s}^i\in \{0,1\}, i=1,2,\ldots,L,~
s=1,\ldots,n+1,
\end{eqnarray}
have the same optimal  solutions $(\gamma_{k+s}^i)^*, i=1,\ldots,L,
s=1,\ldots,n$ respectively. By Lemma \ref{lem_2},  the problem
(\ref{Eqsm_19}) is also equivalent to solving the problem
(\ref{Eqsm_22}). Therefore, if the primal sensor selection problem
(\ref{Eqsm_19}) has an optimal solutions, both the problem
(\ref{Eqsm_19}) and the problem (\ref{Eqsm_22}) can be equivalently
transformed to solve the problem (\ref{Eqsm_2_7}).\Box
\end{proof}

\textbf{The derivation of the problem (\ref{Eqsm_new_4_5}})

By $(\gamma_{k+n}^i)^2=\gamma_{k+n}^i$, the problem
(\ref{Eqsm_new_4_4}) is equivalent to
\begin{eqnarray}%
\nonumber\min_{\gamma_{k+n}^i} ~&&\gamma'\mB\gamma\\
\nonumber\mbox{subject to}~
~ &&\gamma'\diag(\va_p)\gamma ~\unrhd~b_p,~ p=1,\ldots,P,\\
\nonumber&&\gamma_{k+n}^i\in \{0,1\}, i=1,2,\ldots,L, n=1,\ldots,N,
\end{eqnarray}
where $\gamma$ is defined in (\ref{Eqsm_new_3_3}); $\diag(\va_p)$
and $\mB=\diag(\omega_1\mB_{k+1},\ldots, \omega_N\mB_{k+N})$ are
diagonal matrix and diagonal block matrix respectively. If we let
$\tau_{k+n}^i=2\gamma_{k+n}^i-1$  and denote by
$\tau_{k+n}\triangleq(\tau_{k+n}^1,\ldots,\tau_{k+n}^L)'$ and
$\tau\triangleq(\tau_{k+1}',\ldots,\tau_{k+N}')'$, then the problem
is equivalent to
\begin{eqnarray}%
\label{Eqsm_new_3_4_6}\min_{\tau_{k+n}^i} ~&&\frac{1}{4}(\tau+\textbf{1})'\mB(\tau+\textbf{1})\\
\nonumber\mbox{subject to}~
~ &&\frac{1}{4}(\tau+\textbf{1})'\diag(\va_p)(\tau+\textbf{1}) ~\unrhd~b_p,~ p=1,\ldots,P,\\
\nonumber&&(\tau_{k+n}^i)^2=1, i=1,2,\ldots,L, n=1,\ldots,N,
\end{eqnarray}
where  $\textbf{1}$ is a 1 vector with compatible dimensions.
Moreover, it is equivalent to
\begin{eqnarray}%
\label{Eqsm_new_3_4_7}\min_{\tau_{k+n}^i,~t} ~&& (\tau+\textbf{1})'\mB(\tau+\textbf{1})\\
\nonumber&&=(\tau_{k+n}'~~ t)\left(
                       \begin{array}{cc}
                         \mB &\mB\textbf{ 1} \\
                         \textbf{1 }'\mB& 0  \\
                       \end{array}
                     \right)\left(
                              \begin{array}{c}
                                \tau_{k+n} \\
                                t \\
                              \end{array}
                            \right)\\
\nonumber\mbox{subject to}~ ~ &&(\tau'~~ t)\left(
                       \begin{array}{cc}
                         \diag(\va_p) &\textbf{ 1} \\
                         \textbf{1 }'& 0  \\
                       \end{array}
                     \right)\left(
                              \begin{array}{c}
                                \tau \\
                                t \\
                              \end{array}
                            \right)~\unrhd~4b_p-\textbf{ 1}'\diag(\va_p)\textbf{ 1},~ p=1,\ldots,P,\\
\nonumber&&(\tau_{k+n}^i)^2=1, i=1,2,\ldots,L, n=1,\ldots,N,\\
\nonumber&&t^2=1.
\end{eqnarray}
Problem (\ref{Eqsm_new_3_4_6}) is equivalent to
(\ref{Eqsm_new_3_4_7}) in the sense: if $(\tau^*,t^*)$ is the
optimal solution to (\ref{Eqsm_new_3_4_7}), then $\tau^*$
(respectively $-\tau^*$) is an optimal solution to
(\ref{Eqsm_new_3_4_6})  when $t^*=1$ (respectively $t^*=-1$).
Moreover, the problem is equivalent to
\begin{eqnarray}%
\label{Eqsm_new_3_4_8}\min_{\tau,~ t} &&(\tau'~~ t)\mC\left(
                              \begin{array}{c}
                                \tau \\
                                t \\
                              \end{array}
                            \right)\\
\nonumber\mbox{subject to}~ ~ &&(\tau'~~ t)\mE_0\left(
                              \begin{array}{c}
                                \tau \\
                                t \\
                              \end{array}
                            \right)\unrhd~(4b_p-\textbf{ 1}'\diag(\va_p)\textbf{1}),~~~ p=1,\ldots,P,\\
\nonumber&&(\tau'~~ t)\mE_s\left(
                              \begin{array}{c}
                                \tau \\
                                t \\
                              \end{array}
                            \right)=1, ~s=1,\ldots,(NL+1),
\end{eqnarray}
where
\begin{eqnarray}
\nonumber\mC=\left(
                       \begin{array}{cc}
                         \mB &\mB\textbf{ 1} \\
                         \textbf{1 }'\mB&0 \\
                       \end{array}
                     \right),\\
\nonumber\mE_0=\left(
                       \begin{array}{cc}
                         \diag(\va_p) &\textbf{ 1} \\
                         \textbf{1 }'& 0  \\
                       \end{array}
                     \right);
\end{eqnarray}
$\mE_s$ is a matrix with $s$-row and $s$-column  $\mE_s(s,s)=1$,
others  equal 0, for $s=1,\ldots,(NL+1)$. By introducing a new
variable $\mX=(\tau'~~ t)'(\tau'~~ t)$ and removing the constraint
$rank(\mX)=1$,  the  problem (\ref{Eqsm_new_3_4_8}) can be relaxed
to the problem (\ref{Eqsm_new_4_5}).\Box

\section*{Acknowledgment}
We would like to thank Yunmin Zhu  for his helpful suggestions that
greatly improved the quality of this paper.



\begin{figure}[h]
\vbox to 8.5cm{\vfill \hbox to \hsize{\hfill
\scalebox{0.5}[0.5]{\includegraphics{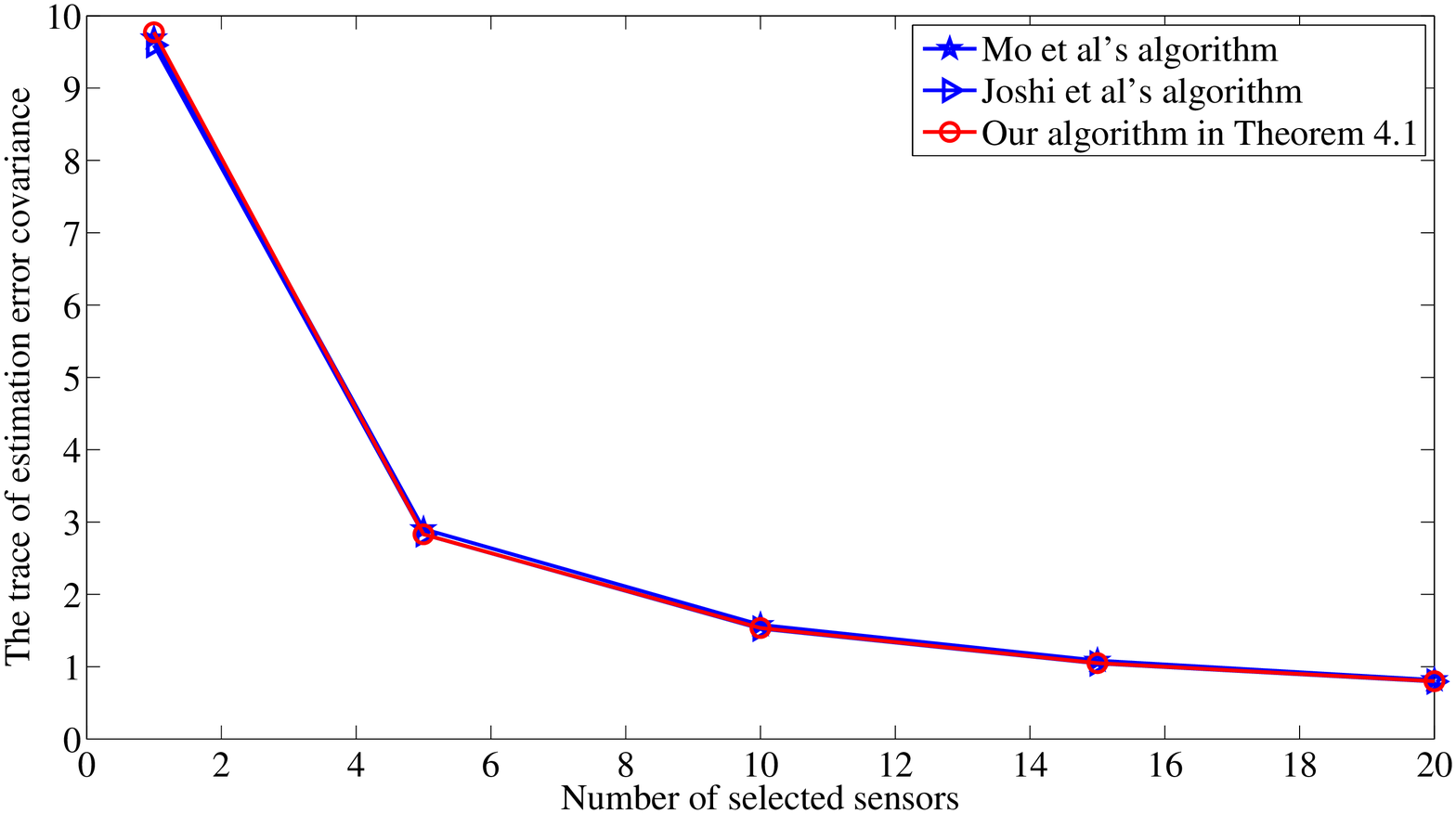}}
\hfill}\vfill}
\caption{The traces of the  estimation error covariance are plotted
as a function of number of selected sensors.}\label{fig_01}
\end{figure}

\begin{figure}[h]
\vbox to 8.5cm{\vfill \hbox to \hsize{\hfill
\scalebox{0.5}[0.5]{\includegraphics{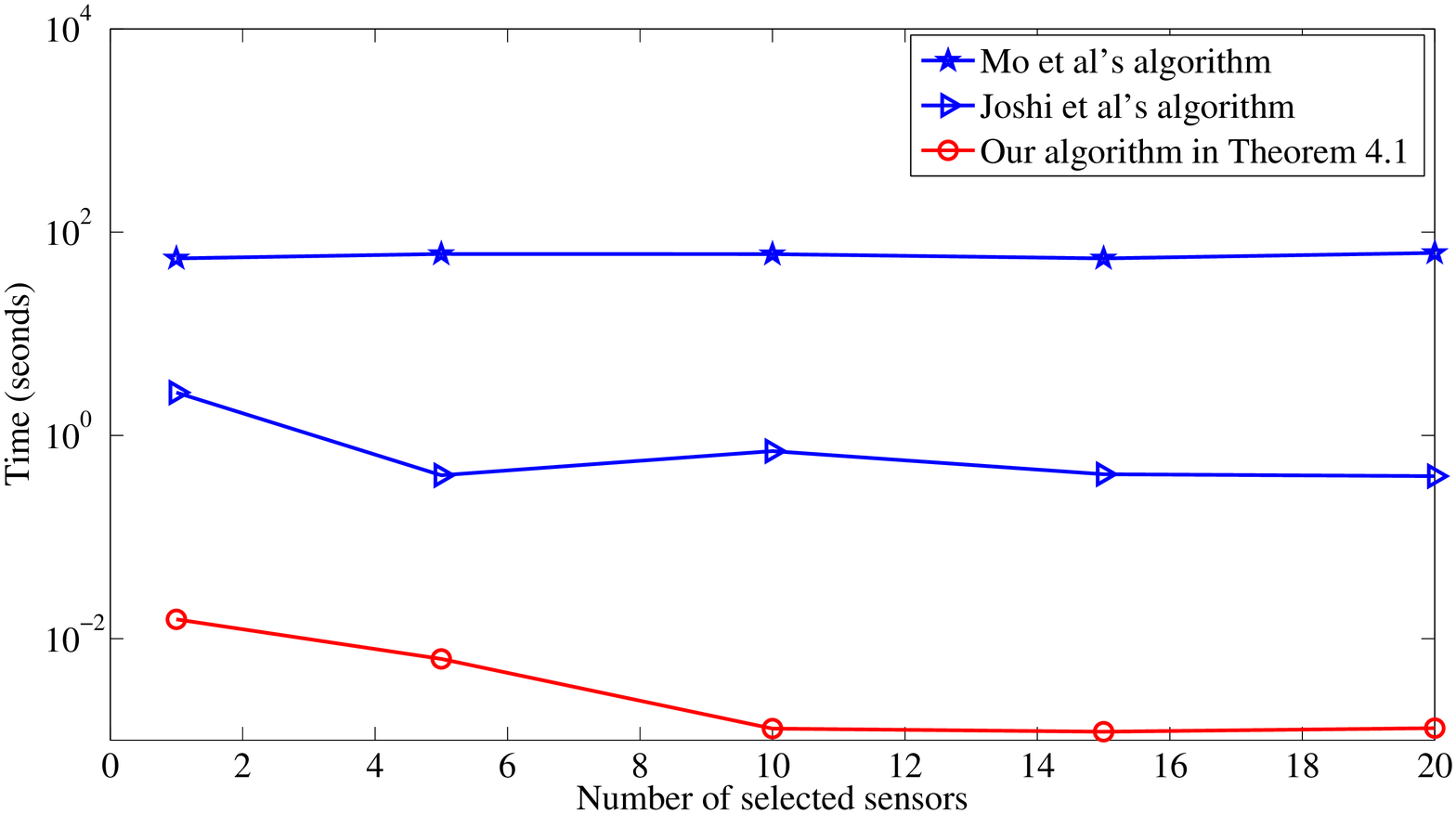}}
\hfill}\vfill}
\caption{The cpu times are plotted as a function of number of
selected sensors. }\label{fig_02}
\end{figure}

\begin{figure}[h]
\vbox to 8.5cm{\vfill \hbox to \hsize{\hfill
\scalebox{0.5}[0.5]{\includegraphics{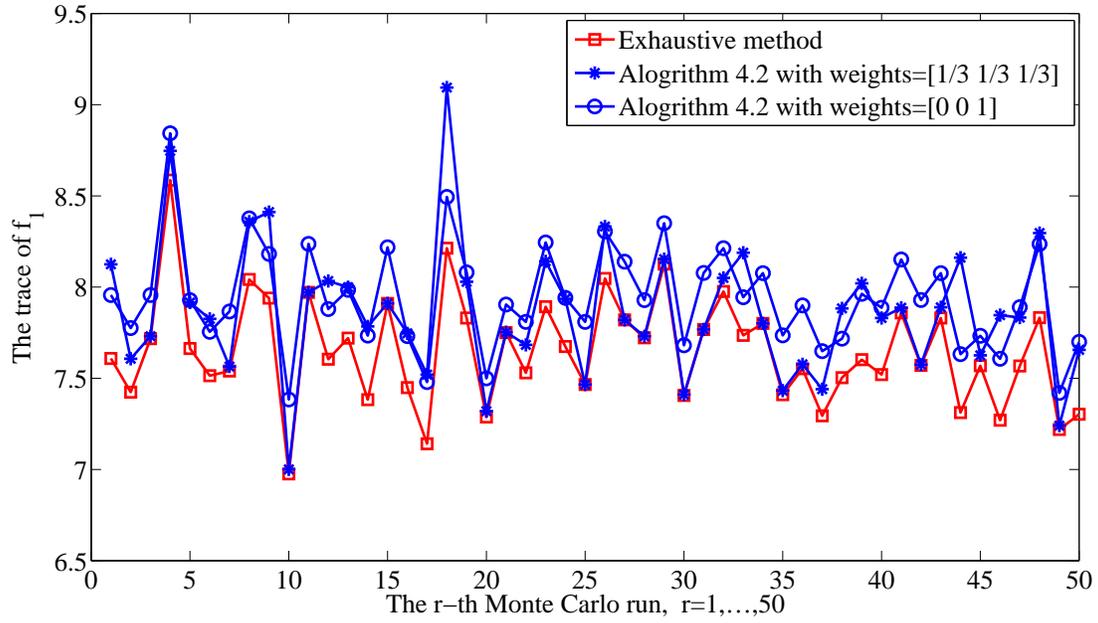}} \hfill}\vfill}
\caption{The traces of the final estimation error covariance for 50
Monte Carlo runs.}\label{fig_1}
\end{figure}

\begin{figure}[h]
\vbox to 8.5cm{\vfill \hbox to \hsize{\hfill
\scalebox{0.5}[0.5]{\includegraphics{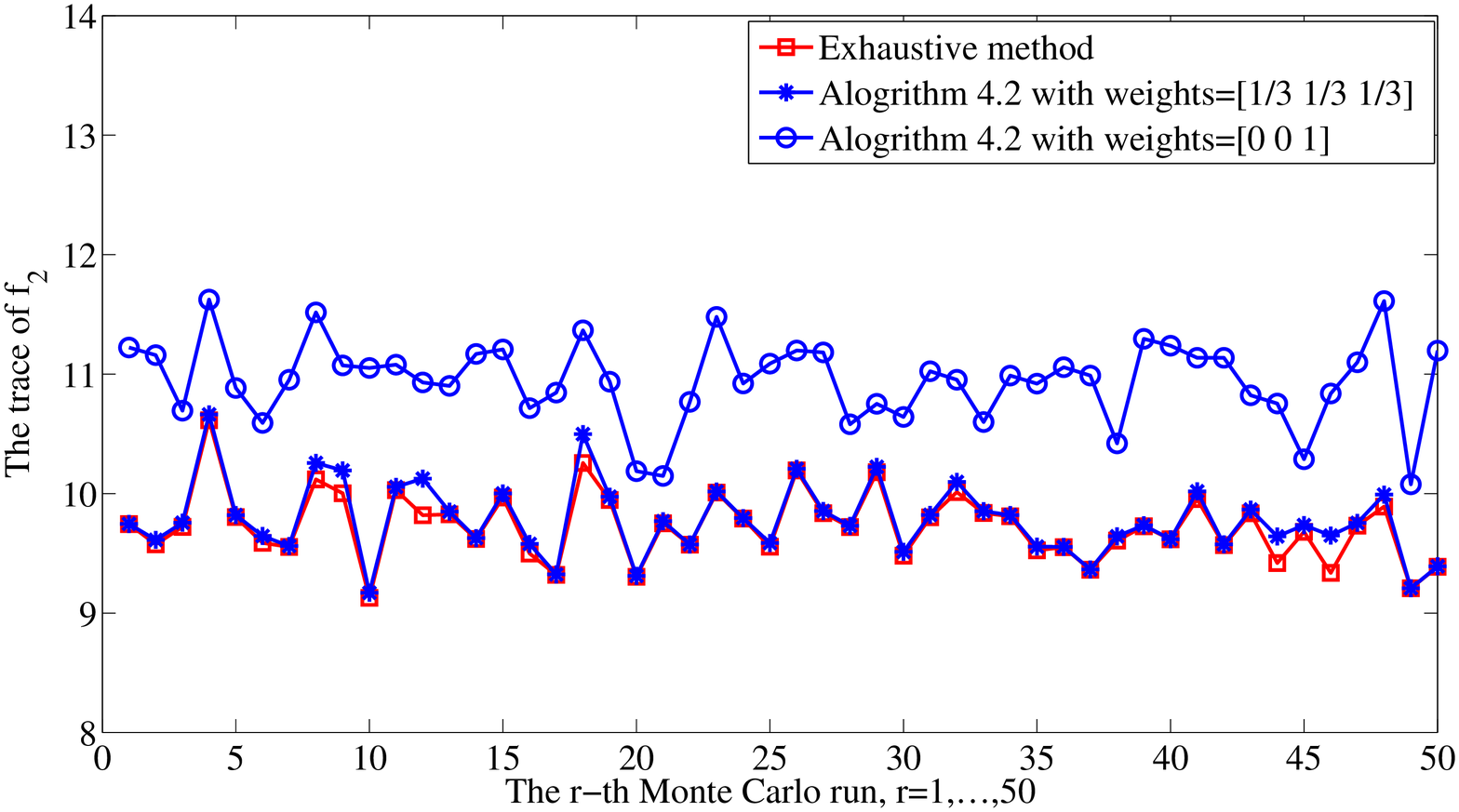}} \hfill}\vfill}
\caption{The traces of the average estimation error covariance in 50
Monte Carlo runs. }\label{fig_2}
\end{figure}

\begin{figure}[h]
\vbox to 8.5cm{\vfill \hbox to \hsize{\hfill
\scalebox{0.5}[0.5]{\includegraphics{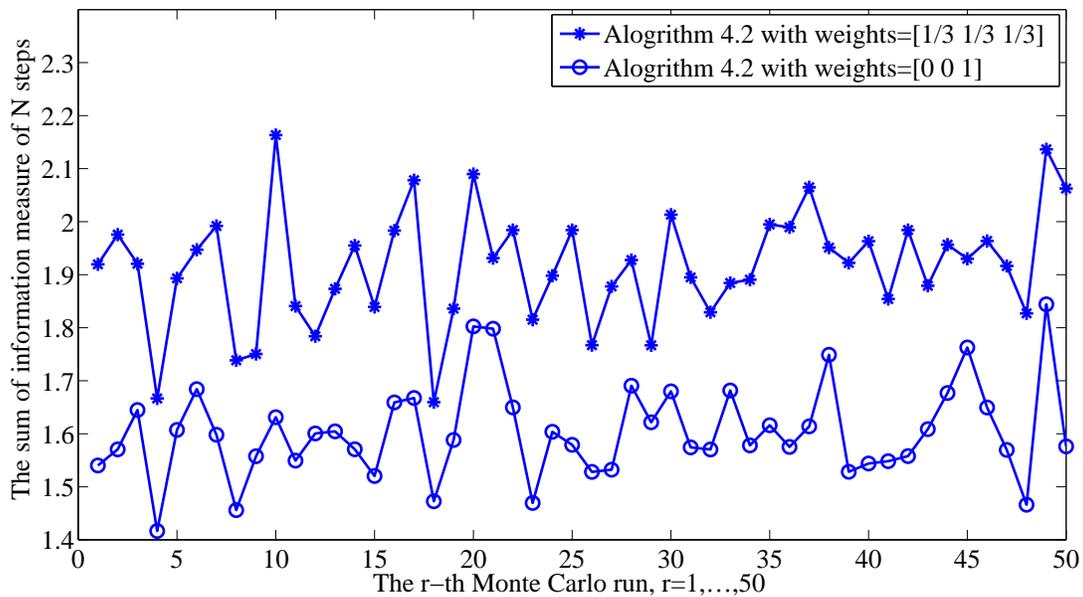}}
\hfill}\vfill}
\caption{The sum of information measures of $N$ time steps for 50
Monte Carlo runs}\label{fig_3}
\end{figure}

\begin{figure}[h]
\vbox to 8.5cm{\vfill \hbox to \hsize{\hfill
\scalebox{0.5}[0.5]{\includegraphics{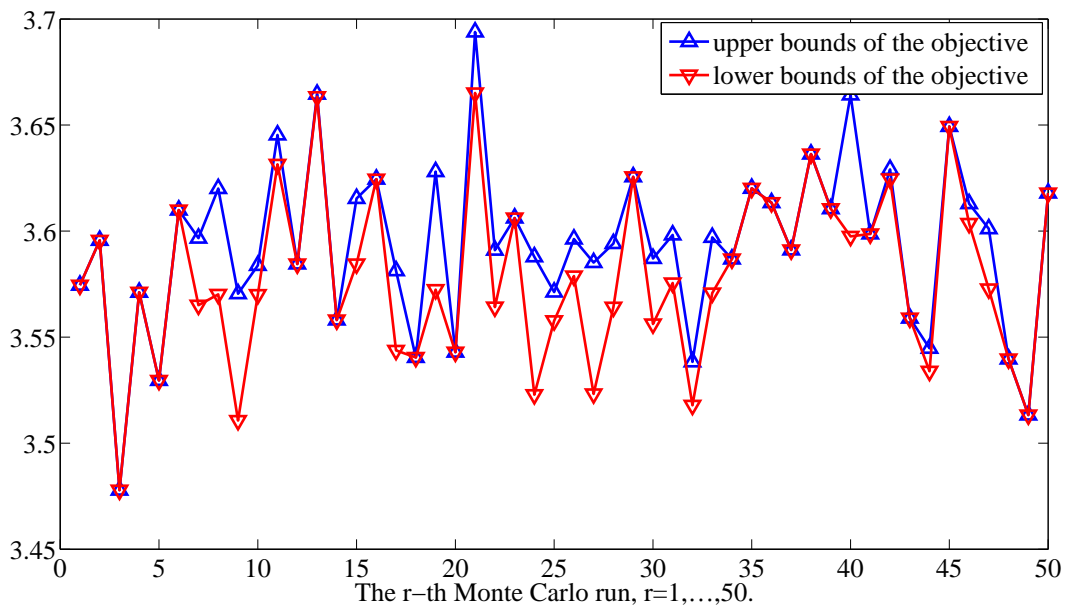}}
\hfill}\vfill}
\caption{The upper bound and lower bound of the objective function
of the optimization problem (\ref{Eqsm_new_3_7}) for 50 Monte Carlo
runs. }\label{fig_4}
\end{figure}

\begin{figure}[h]
\vbox to 8.5cm{\vfill \hbox to \hsize{\hfill
\scalebox{0.5}[0.5]{\includegraphics{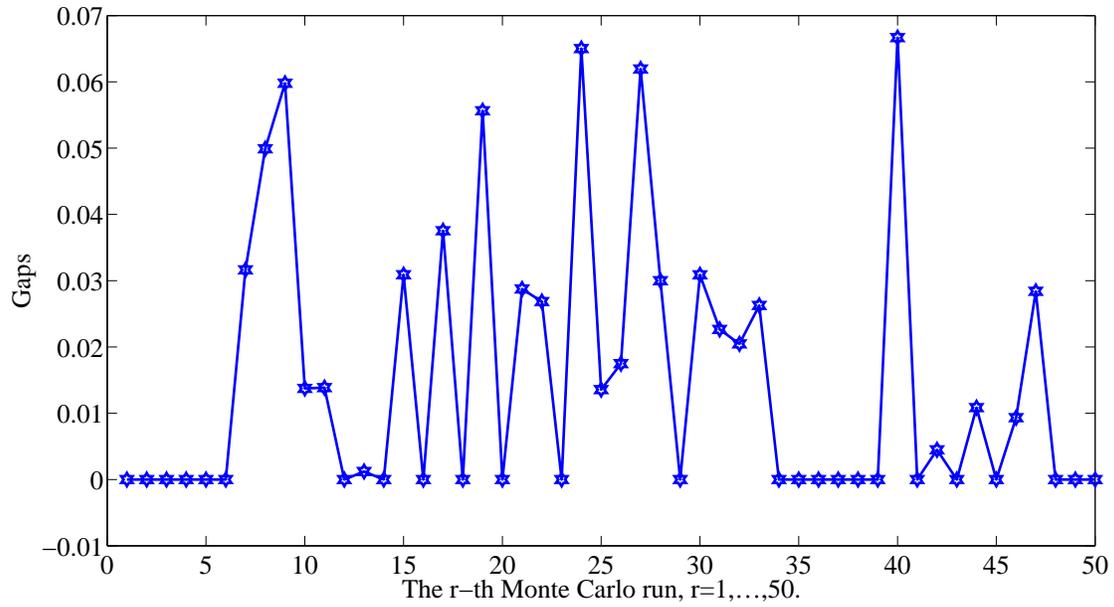}} \hfill}\vfill}
\caption{The gaps, i.e. the upper bounds minus the lower bounds
shown in  Figure \ref{fig_4}. }\label{fig_5}
\end{figure}

\begin{figure}[h]
\vbox to 8.5cm{\vfill \hbox to \hsize{\hfill
\scalebox{0.5}[0.5]{\includegraphics{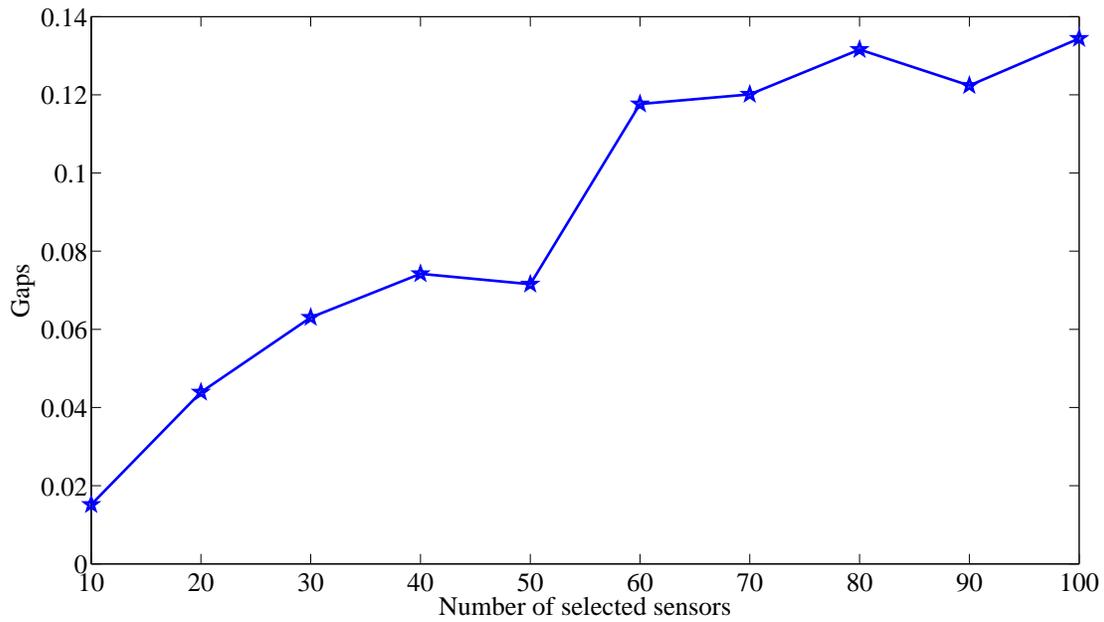}}
\hfill}\vfill}
\caption{The average gaps based on 50 Monte Carlo runs are plotted
as a function of number of selected sensors. }\label{fig_05}
\end{figure}

\begin{figure}[h]
\vbox to 8.5cm{\vfill \hbox to \hsize{\hfill
\scalebox{0.5}[0.5]{\includegraphics{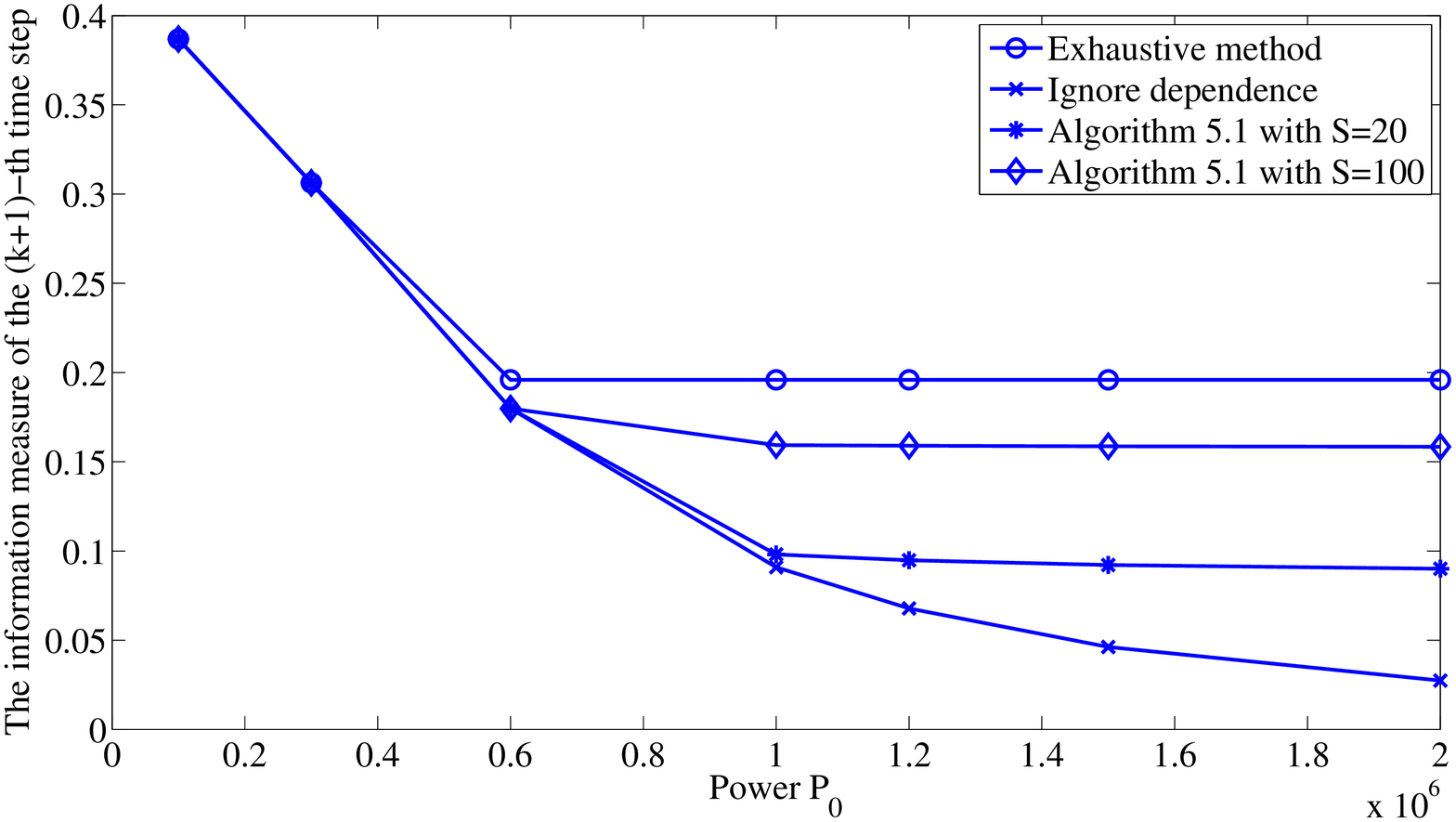}} \hfill}\vfill}
\caption{Information measure  based on the selected sensors at the
$(k+1)$-th time step from weak to strong signal power of the jammer
(from weak to strong correlation between sensors). }\label{fig_6}
\end{figure}

\begin{figure}[h]
\vbox to 8.5cm{\vfill \hbox to \hsize{\hfill
\scalebox{0.5}[0.5]{\includegraphics{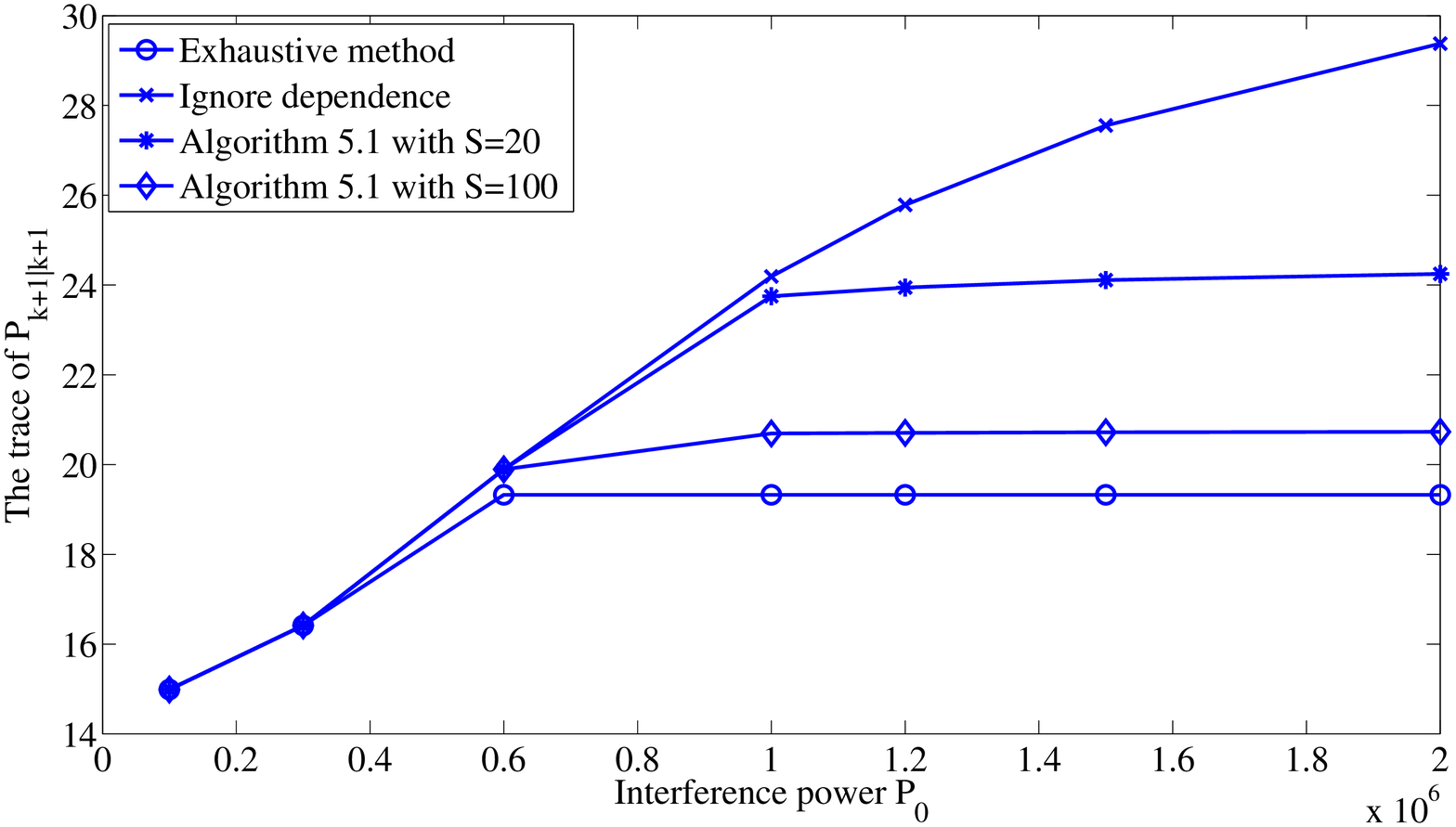}} \hfill}\vfill}
\caption{The trace of estimation error covariance at the $(k+1)$-th
time step from weak  to strong signal power of the jammer (from weak
to strong correlation between sensors). }\label{fig_7}
\end{figure}

\begin{figure}[h]
\vbox to 8.5cm{\vfill \hbox to \hsize{\hfill
\scalebox{0.5}[0.5]{\includegraphics{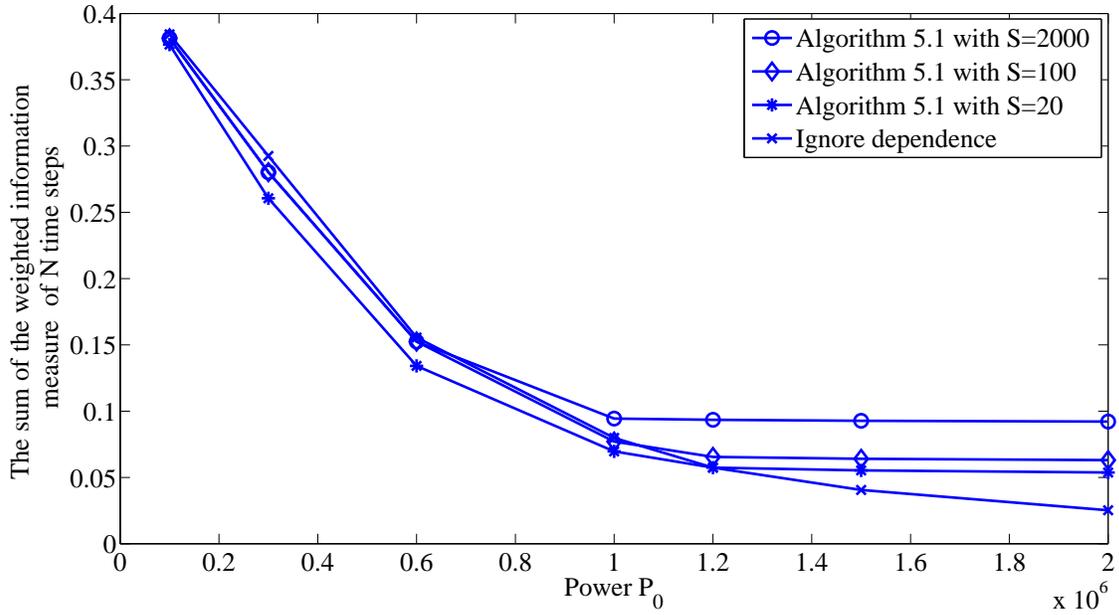}}
\hfill}\vfill}
\caption{The sum of the weighted information measures of $N$ time
steps from weak  to strong signal power of the jammer (from weak to
strong correlation between sensors). }\label{fig_8}
\end{figure}

\begin{figure}[h]
\vbox to 8.5cm{\vfill \hbox to \hsize{\hfill
\scalebox{0.5}[0.5]{\includegraphics{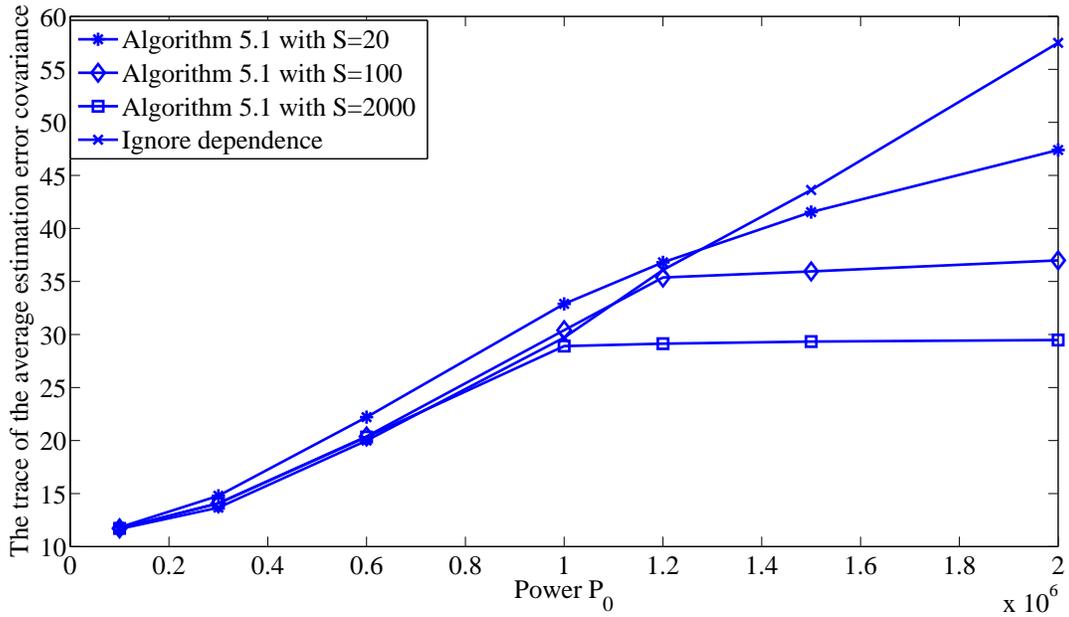}}
\hfill}\vfill}
\caption{The trace of the average estimation error covariance of $N$
time steps from weak  to strong signal power of the jammer (from
weak to strong correlation between sensors).  }\label{fig_9}
\end{figure}

\begin{figure}[h]
\vbox to 8.5cm{\vfill \hbox to \hsize{\hfill
\scalebox{0.5}[0.5]{\includegraphics{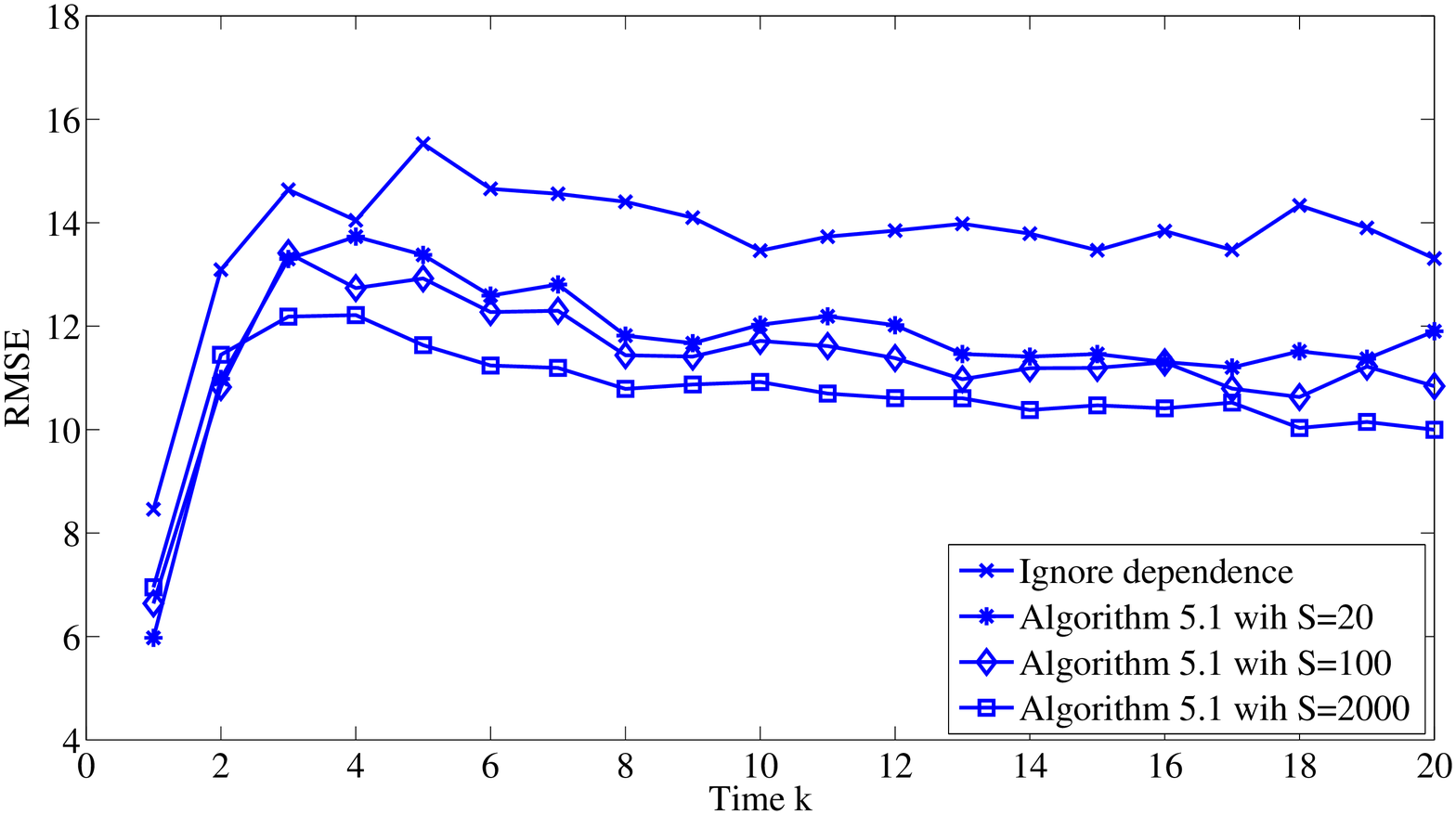}} \hfill}\vfill}
\caption{RMSE of the state estimates based on 200 Monte Carlo runs.
}\label{fig_10}
\end{figure}

\begin{figure}[h]
\vbox to 8.5cm{\vfill \hbox to \hsize{\hfill
\scalebox{0.5}[0.5]{\includegraphics{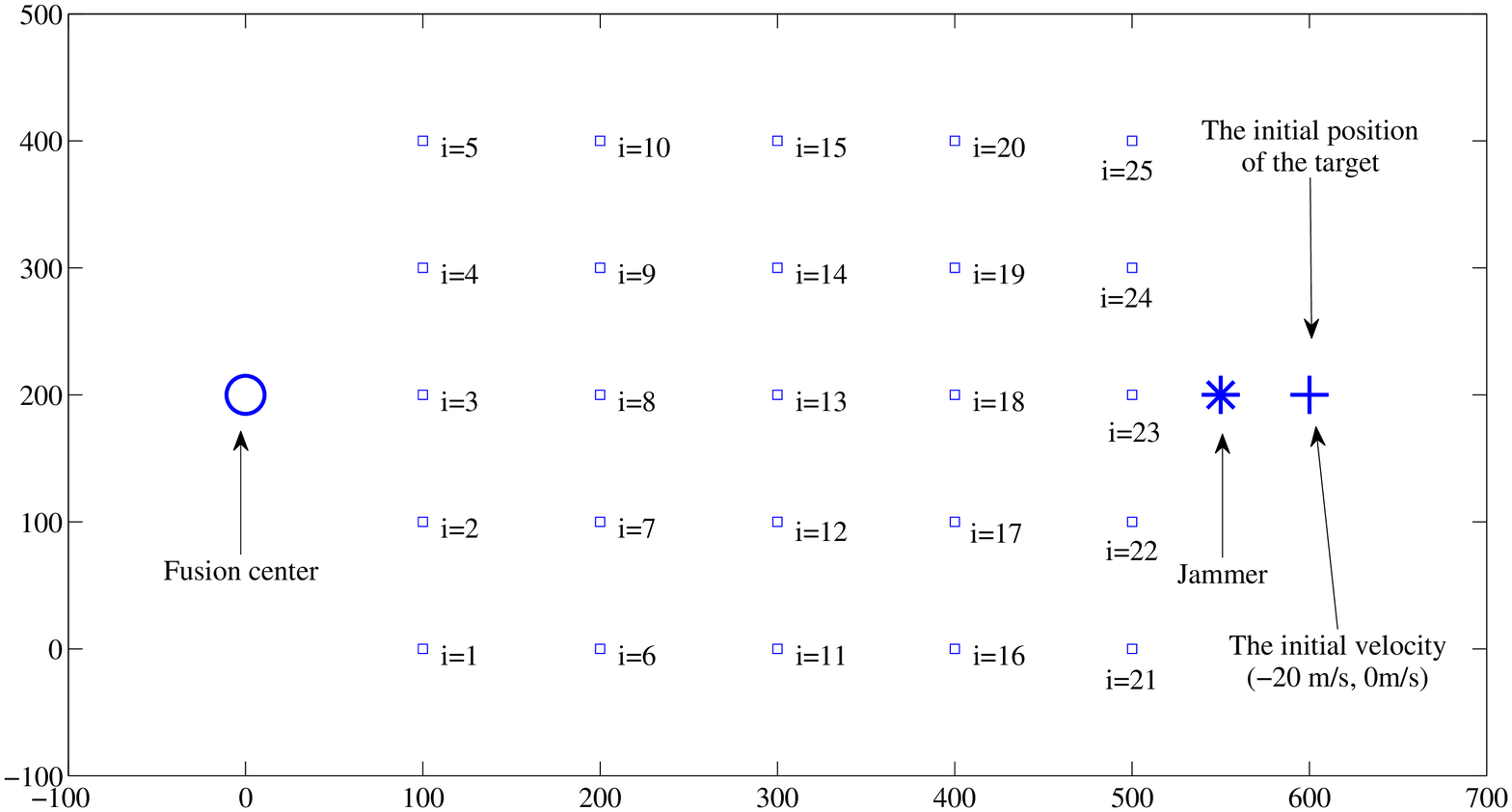}}
\hfill}\vfill}
\caption{The sensor network with a jammer}\label{fig_0}
\end{figure}

\end{document}